\documentclass[12pt]{article}
\usepackage{a4wide}
\usepackage{amsmath,amsfonts,amssymb,
amsthm,dsfont,color, MnSymbol,relsize}
\usepackage{hyperref}
\usepackage{csquotes}
\usepackage{mdframed}

\usepackage{dingbat}

\usepackage{multicol}

\usepackage[active]{srcltx}
\usepackage{dutchcal}
\usepackage{mathrsfs}
\usepackage{amscd,graphicx}
\usepackage[shortlabels]{enumitem}

\usepackage[cmtip,arrow]{xy}
\usepackage{pb-diagram,pb-xy}

\usepackage[nottoc]{tocbibind}
\newtheorem{theorem}{Theorem}[section]
\newtheorem{definition}[theorem]{Definition}
\newtheorem{proposition}[theorem]{Proposition}
\newtheorem{corollary}[theorem]{Corollary}

\theoremstyle{definition}
\newtheorem{remark}[theorem]{Remark}

\newtheorem{notation}[theorem]{Notation}

\newtheorem{hypothesis}[theorem]{Hypothesis}
\newtheorem{conclusion}[theorem]{Conclusion}

\newcommand\numberthis{\addtocounter{equation}{1}\tag{\theequation}}

\def\e{\mathfrak{e}}

\def\R{\mathbb{R}} 
\def\Rd{\mathbb{R}^d} 
\def\Z{\mathbb{Z}} 
\def\Zd{\mathbb{Z}^d} 
\def\Sb{\mathbb{S}}

\def\Co{{\mathbb C}} 
 
\def\Op{\mathfrak{Op}} 

\def\X{\mathcal X}
\def\U{\mathcal{U}}

\def\cL{\mathcal{L}}
\def\H{\mathcal{H}}
\def\E{\mathcal E}
\def\T{\mathcal{Q}}

\def\id{{\rm id\hspace*{-1.5pt}l}}
\def\bb1{{\rm{1}\hspace{-3pt}\mathbf{l}}}
\def\Ie0{[-\epsilon_0,\epsilon_0]}

\def\Id{\text{I}\hspace*{-2pt}\text{I{\sf d}}}

\def\Int{\mathfrak{I}\mathit{nt}}
\def\supp{\mathop{\rm supp} \nolimits}

\def\p{\mathfrak{p}}
\def\s{\mathfrak{s}}
\def\bz{\mathfrak{z}\hspace*{-4pt}\mathfrak{z}}

\def\dd{\mathfrak{d}}

\def\plus2{\hat{+}_{\hspace*{-2pt}_{2\Gamma}}}

\def\z{\mathfrak{z}}
\def\zz{\mathcal{z}}

\def\BC2{\mathbb{B}\big(\mathbb{C}^2\big)}

\def\repi{\rightY\hspace*{-4pt}\rightarrow}

\def\L2T{L^2(\mathbb{T})}

\def\bigprod{{\text{\large{$\prod$}}}}

\def\tgamma{\tilde{\gamma}}

\def\tB{\widetilde{B}}
\def\tA{\widetilde{A}}
\def\tmO{\widetilde{\mO}}
\def\tbz{\tilde{\bz}}
\def\Nb{\mathbb{N}_{\bullet}}

\def\beq{\begin{equation}}
\def\eeq{\end{equation}}

\def\tf{\widetilde{f}}
\def\tg{\widetilde{g}}

\def\F0{\mathlarger{\mathlarger{\mathbf{\Lambda}}}}
\def\Fb{\mathlarger{\mathlarger{\mathbf{\Lambda}}}_{\text{\tt bd}}}
\def\Fp{\mathlarger{\mathlarger{\mathbf{\Lambda}}}_{\text{\tt pol}}}
\def\z{\mathfrak{z}}
\def\G{\mathcal{G}}

\def\supp{\text{\tt supp}}

\def\Sd{\mathbb{S}^d}
\def\mbz{\mathring{\bz}}

\def\mf{\mathring{f}}
\def\mO{\mathfrak{O}}
\def\mI{\mathcal{I}}

\numberwithin{equation}{section}
\let\oldllcorner\llcorner
\renewcommand{\llcorner}[1][0pt]{%
  \mathrel{\raisebox{#1}{\text{\LARGE{$\oldllcorner$}}}}}

\def\lnu{\mathlarger{\nu}}

\title{The fibre operators in the Bloch-Floquet decomposition\\of periodic magnetic pseudo-differential operators}

\author{Horia D. Cornean\footnote{Department of Mathematical Sciences, Aalborg University, Thomas Manns Vej 23, 9220 Aalborg, Denmark; cornean@math.aau.dk, ORCID iD 0000-0003-2700-8785}, Bernard Helffer\footnote{Laboratoire de Math{\'e}matiques Jean Leray,  Nantes Universit{\'e}  and CNRS, Nantes, France;
Bernard.Helffer@univ-nantes.fr, ORCID iD 0009-0004-4238-015X`}, Radu Purice\footnote{\enquote{Simion Stoilow} Institute of Mathematics of the Romanian Academy, P.O. Box 1-764, 014700 Bucharest, Romania; Radu.Purice@imar.ro, ORCID iD 0000-0002-9012-7982}}

\begin{document}
\maketitle

\begin{abstract}
We study the structure of the fibre operators corresponding to  periodic magnetic pseudo-differential operators having periodic magnetic potentials. We obtain explicit formulas for their
distribution kernel, both when these fibres are seen as operators on the $d$-dimensional torus,
and also when they are seen as infinite matrices acting on a discrete $\ell^2$ space via a 
discrete Fourier transform. Moreover, using these distribution kernels we prove that the fibre
operators are toroidal pseudo-differential operators.
\end{abstract}

%\tableofcontents

\section{Introduction}

It  is well-known that any self-adjoint operator in $L^2(\Rd)$ that is left invariant by translations with elements of the subgroup $\Zd $ is unitarily equivalent to a direct integral of fibre operators acting on the quotient $\Rd/\Zd$ that is diffeomorphic to the $d$-dimensional torus. An interesting question is to understand the structure of these fibre operators when dealing with \textit{periodic pseudo-differential operators} and mainly with \textit{periodic magnetic pseudo-differential operators}. While, for periodic differential operators the situation has been thoroughly studied (see for example \cite{Ku-1, Ku-2} and references therein), the case of pseudo-differential operators is only mentioned in \cite{Ku-1} and some results are given in \cite{HM,Mo}. Meanwhile, as our analysis in \cite{CHP-1}-\cite{CHP-4} has shown, a good understanding of the structure of these fibre operators may be necessary when dealing with several types of problems and we devote this note to this subject.

{
Let us also mention the connection with other general versions of pseudo-differential calculus on different classes of groups \cite{P,BB,MR,FRR}. The construction of a ``symmetric Weyl calculus" still remains of interest even for Pontryagin dual pairs of abelian groups, see for example  \cite{W} and notice that the case of the torus, in which we are interested, is not covered, being outside the class of ``odd groups".}

\subsection{Some notations}\label{SS-notations}
\begin{enumerate}[{\rm a]}]
	\item $\Nb:=\mathbb{N}\setminus\{0\}$
	\item For any real number $t\in\R$ we define its integer  part 
	$
	\lfloor t\rfloor:=\max\{k\in\Z\,,\,k\leq t\}\in\Z\,.$
	\item Given any element $a$ in a normed space, we shall use the notation 
	$
	<a>:=\sqrt{1+|a|^2}\,.
	$
	\item Given two linear topological spaces $\mathcal{T}_j$ (with $j=1,2$) we denote by $\mathcal{L}\big(\mathcal{T}_1;\mathcal{T}_2\big)$ the space of continuous linear operators $\mathcal{T}_1\rightarrow\mathcal{T}_2$, endowed with the topology of uniform convergence on bounded subsets.
	\item Given a complex Hilbert space $\mathcal{H}$ we shall denote by $\mathbb{B}(\mathcal{H})$ the $C^*$-algebra of bounded linear operators on $\mathcal{H}$ and by $\mathbb{U}\big(\mathcal{H}\big)$ the group of unitary operators on $\mathcal{H}$.
	\item Given a linear continuous operator $T\in\mathcal{L}\big(\mathscr{S}(\X);\mathscr{S}^\prime(\X)\big)$ (remark that any $T\in\mathbb{B}\big(L^2(\X)\big)$ defines by restriction such an operator), by the nuclear theorem it admits a distribution kernel $\mathfrak{K}_T\in\mathscr{S}^\prime(\X\times\X)$ such that $\big(\psi,T\phi\big)_{L^2(\X)}=\big\langle\mathfrak{K}_T\,,\,\overline{\psi}\otimes\phi\big\rangle_{\mathscr{S}(\X\times\X)}$ and we write that $T=\Int_\X\mathfrak{K}_T$.
	\item Given $N\in\mathbb{N}\setminus\{0\}$ we denote by $\underline{N}:=\{j\in\mathbb{N}\setminus\{0\},\ 1\leq j\leq N\}$. Given an algebra $\mathcal{A}$, for any $a\in\mathbb{N}^N$ and any $\mathbf{A}=(A_1,\ldots,A_N)\in\mathcal{A}^N$ we write:
	$$\mathbf{A}^a\,:=\,\underset{1\leq n\leq N}{\bigprod}A_n^{a_n};\qquad|a|\,:=\,\underset{1\leq n\leq N}{\sum}a_n.$$
	\item For each $j\in\{1,\ldots,d\}$ we define the multi-index $\varepsilon_j\in\mathbb{N}^d$ with components $[\varepsilon_j]_l:=\delta_{jl}$.
\end{enumerate}

\subsection{The framework} 

We work in dimension $d\geq2$ and denote by $\X\cong\Rd$ the \textit{configuration space} with the canonical orthonormal basis $\big\{\e_1,\ldots,\e_d\big\}$ and an associated regular lattice $\Gamma:=\underset{1\leq j\leq d}{\bigoplus}\Z\e_j\subset\X$. In order to easily distinguish the configuration space from the \textit{momentum space} defined as its dual, we shall use the notation $\X\cong\R^d$ for the first one and $\X^*\cong\Rd$ for the second one. Then $\Xi:=\X\times\X^*$ is the phase space and we use notations of the form $X:=(x,\xi)\in\Xi$, $Y:=(y,\eta)\in\Xi$, ... . Let  $\Gamma_*:=\underset{1\leq j\leq d}{\bigoplus}\Z\e^*_j\subset\X^*$ be the \textit{dual lattice} generated by \textit{the dual basis} $\{\e^*_j,\ 1\leq j\leq d\}\subset\X^*$ defined by $<\e^*_j\,,\,\e_k>=(2\pi)\delta_{j,k}$. We choose to simplify some formulas by inserting the factor $2\pi$ in the duality map $<\cdot,\cdot>:\X^*\times\X\rightarrow\Co$ and thus for a decomposition $\xi=\underset{1\leq j \leq d}{\sum}\xi_j\e^*_j$ we have the equality $<\xi,x>=2\pi\underset{1\leq j \leq d}{\sum}\xi_j\,x_j$. This leads to some modifications in the formula of the Fourier unitary operators and the fact that the elementary cells in $\Gamma$ and $\Gamma_*$ have volume 1. We shall denote by $\tau_x$ for $x\in\X$ and $\tau_\xi$ for $\xi\in\X^*$ the usual translation operators with $x\in\X$ and resp. $\xi\in\X^*$ on the given linear spaces. We shall use the standard notations for test functions and distributions on $\Rd$. We shall denote by $\mathscr{E}_{\Gamma}(\Rd):=\big\{F\in\mathscr{E}(\Rd)\,,\,F\circ\tau_\gamma=F,\,\forall \gamma\in\Gamma\big\}$, as a subspace of $BC^\infty(\X)$ with the induced Fr\'{e}chet topology.

\begin{definition}\label{D-G-psrtunit}
We call  \emph{$\Zd$-partition of unity} on $\Rd$, any family $\big\{\rho\circ\tau_{\gamma}\big\}_{\gamma\in\Zd}$, where the function $\rho\in C^\infty_0(\Rd;\mathbb{R}_+)$ satisfies the relation $\underset{\gamma\in\Zd}{\sum}\rho(x+\gamma)=1$ for any $x\in\Rd$. Having in view the two lattices $\Gamma\cong\Zd$ in $\X$ and $\Gamma_*\cong\Zd$ in $\X^*$, we shall work with $\Gamma$-partitions of unity on $\X$ and $\Gamma_*$-partitions of unity on $\X^*$.
\end{definition}
We recall that given any abelian locally compact group, its \textit{Pontryagin dual} is the family of its irreducible unitary representations (and thus 1-dimensional) that has a natural topological group structure (see \cite{Fo-AHA}) and also the well-known identification of the dual $[\Rd]^*$ with the Pontryagin dual $\widehat{\Rd}$ given by
 $[\Rd]^*\ni\xi\overset{\sim}{\longrightarrow}\chi_{\xi}\in\widehat{\Rd}\,, \mbox{  with } \chi_{\xi}(x):=e^{-i<\xi,x>}  \mbox{ for } x\in\Rd$.
 We use the symbol $\overset{\sim}{\longrightarrow}$ for bijective maps.
 
 We recall that \textit{a Weyl projective representation} associated with  the pair $(\X,\X^*)$ is a map $W:\X\times\X^*\rightarrow\mathbb{U}(\H)$ defined on a Hilbert space $\H$ such that the restrictions $W(x,\cdot):\X^*\rightarrow\mathbb{U}(\H)$ and $W(\cdot,\xi):\X\rightarrow\mathbb{U}(\H)$ are strongly continuous unitary representations for any $x\in\X$, resp. any $\xi\in\X^*$ and verifying the commutation relations:
\beq\label{F-10}
W(x,\xi)W(y,\eta)=e^{i[<\xi,y>-<\eta,x>]}W(y,\eta)W(x,\xi),\quad\forall\big((x,\xi),(y,\eta)\big)\in(\X\times\X^*)^2.
\eeq
We work with the \textit{Schr\"{o}dinger projective representation} associated with  $(\X,\X^*)$ which is defined on $\H=L^2(\X)$ by the formula:
\beq\label{F-W-R-prepr}
\big(W_{\X}(z,\zeta)f\big)(x):=e^{(i/2)<\zeta,z>}\,e^{-i<\zeta,x+z>}\,f(x+z),\quad\forall{f}\in{L}^2(\X).
\eeq

The functional calculus associated with this representation of $(\X,\X^*)$ is the well-known \textit{Weyl calculus} (see \cite{Ho-3}) which  is a symmetric pseudo-differential calculus (i.e. taking complex conjugate symbols to adjoint operators):
\beq\begin{split}\label{D-Op}
	\Op:\ &\mathscr{S}(\Xi)\rightarrow\mathcal{L}\big(\mathscr{S}^\prime(\X);\mathscr{S}(\X)\big),\\
	\big(\Op(\Phi)\psi\big)(x)&=\int_{\X}\hspace*{-0.2cm}dy\int_{\X^*}\hspace*{-0.3cm}d\zeta\,e^{i<\zeta,x-y>}\,\Phi\big((x+y)/2,\zeta\big)\,\psi(y),\ \forall\psi\in\mathscr{S}(\X)\,,
\end{split}\eeq
where the formula in the second line may be extended to $\psi\in\mathscr{S}^\prime(\X)$ by duality.

\paragraph{The pair of $d$-dimensional tori}
Let us introduce the circle  $\Sb:=\{\z\in\Co,\ |\z|=1\}$ with its structure of an abelian Lie group with the multiplication induced from $\Co$ and the usual normalized Haar measure (giving it the measure 1). Then we have the isomorphism $\R/\Z\cong\Sb$ and the canonical quotient projection $\R\ni{t}\mapsto\mathfrak{e}(t):={e^{2\pi{i}t}}\in\Sb$. We  always identify $\Zd$-periodic functions on $\Rd$ with functions defined on $\Rd/\Zd$ and thus on the $d$-dimensional torus. Let us introduce the following two objects playing a central role in our work, namely the two quotient groups isomorphic to $d$-dimensional tori: 
\beq\label{D-T-Tstar}
\T:=\X/\Gamma\cong\Sb^d,\qquad\T_*:=\X^*/\Gamma_*\cong\Sb^d
\eeq
where the isomorphisms are defined by the two canonical orthonormal basis $\{\e_j\}_{j=1}^d\subset\X$ and $\{\e^*_j\}_{j=1}^d\subset\X^*$. We shall work with the quotient projections $\p:\X\repi\T$ and $ \p_*:\X^*\repi\T_*$ and 
 we shall denote the identity in $\T_*$ by:
 \beq\label{DF-1}
\mathfrak{i}^*\in\T_*:\quad\mathfrak{i}^*\bz^*=\bz^*,\ \forall\bz^*\in\T_*.
\eeq

 The Lie group structure of $\Sb^d$ implies the existence of a set of $d$ commuting invariant differential operators. On $\T\cong\Sb^d$ we will denote them by $\mathring{\partial}_j$ with $1\leq j\leq d$, with $\mathring{\Delta}:=\underset{1\leq j\leq d}{\sum}[\mathring{\partial}_j]^2$. They are chosen in order to satisfy the identities: 
 \beq\label{DF-T-deriv}
 \partial_j(\varphi\circ\p)=(\mathring{\partial}_j\varphi)\circ\p,\qquad\forall\varphi\in{C}^\infty(\T).
 \eeq

Let us consider the unit cells $\E:=\{\hat{x}\in\X,\ \hat{x}_j\in[-1/2,1/2),\,1\leq j\leq d\}\subset\X$ and $\E_*:=\{\hat{\xi}\in\X^*,\ \hat{\xi}_j\in[-1/2,1/2),\,1\leq j\leq d\}\subset\X^*$and the 'local' inverses $\s:\T\rightarrow\E$ and $\s_*:\T_*\rightarrow\E_*$ verifying $\p\circ\s=\id_{\T}$ and $\p_*\circ\s_*=\id_{\T_*}$. We use the notation $\bz:=(\z_1,\ldots\z_d)\in\T$ and $\bz^*:=(\z^*_1,\ldots,\z^*_d)\in\T_*$ for the group elements. Finally we shall denote by $d\bz$ and $d\bz^*$ the normalized Haar measures on $\T$ and respectively on $\T_*$ noticing that they are in fact equal with the inverse measure of $d\hat{x}$ on $\E$ through $\s$ and resp of $d\hat{\xi}$ on $\E_*$ through $\s_*$. For the characters we shall use the identifications: $\T_*\ni\bz^*\mapsto\chi_{\bz^*}\in\widehat{\Gamma}$ and $\Gamma\ni\gamma\mapsto\hat{\chi}_{\gamma}\in\widehat{\T_*}$  with:
\beq \label{DF-char-zst}
\hat{\chi}_{\bz^*}(\gamma):=(\bz^*)^{\gamma}\,, \, \hat{\chi}_{\gamma}(\bz^*)= (\bz^*)^{-\gamma}\,,\quad\forall(\bz^*,\gamma)\in\T_*\times\Gamma.
\eeq

\paragraph{The magnetic field.} Let $\Fb^p(\X)$ the space of smooth p-forms on $\X$ that have all the components of class $BC^\infty(\X)$ and by $\Fp^p(\X)$ the space of smooth p-forms on $\X$ that are of class $C^\infty_{\text{\tt pol}}(\X)$. A \textit{regular magnetic field} on $\X$ is described as \textit{a closed 2-form} $B\in\Fb^2(\X)$. Due to the contractibility of the space $\X$, $B$ is also exact and thus there exists a 1-form $A\in\Fp^1(\X)$ such that $B=dA$. Then the magnetic pseudo-differential calculus (\cite{MP-1,IMP-1}-\cite{IMP-3}) proposes the definition of a twisted quantization associated with  the vector potential $A\in\Fp^1(\X)$:
\beq\begin{split}\label{D-OpA}
	&\Op^A:\ \mathscr{S}(\Xi)\rightarrow\mathcal{L}\big(\mathscr{S}^\prime(\X);\mathscr{S}(\X)\big),\qquad\\
	&\big(\Op^A(\Phi)\psi\big)(x)=\int_{\X}\hspace*{-0.2cm}dy\int_{\X^*}\hspace*{-0.3cm}d\zeta\,\Lambda^A(x,y)\,e^{i<\zeta,x-y>}\,\Phi\big((x+y)/2,\zeta\big)\,\psi(y),\ \forall\psi\in\mathscr{S}(\X)\,,\\
	&\Lambda^A(x,y):=\exp\Big(-i\int_{[x,y]}A\Big)\,,
\end{split}\eeq
with a similar extension to $\mathscr{S}^\prime(\X)$ by duality.  

The map $\Op$ in \eqref{D-Op} equals the map $\Op^0$ in \eqref{D-OpA}. The main difficulty in implementing the above 'magnetic' quantization, comes from the fact that for a magnetic field that does not vanish at infinity, as for example for any non-zero constant magnetic field, the vector potential is growing at infinity and thus the prescription in \eqref{D-OpA} does not fit well in the frame of an usual Weyl calculus with modified symbols. In this paper we shall deal only with a class of periodic magnetic fields that admit periodic vector potentials and as a consequence belong to $\Fb^1(\X)$ (see Theorem \ref{L-Abounded}) but we have in view further applications to situations with superpositions of different magnetic fields.

	In \cite{IMP-1} it is proven that for a closed 2-form $B\in\Fb^2(\X)$ the operator $\Op^A:\mathscr{S}(\Xi)\rightarrow\cL \big(\mathscr{S}^\prime(\X);\mathscr{S}(\X)\big)$ extends to a continuous map: $S^p_1(\Xi)\rightarrow\cL \big(\mathscr{S}(\X)\big)$ for any $p\in\R$. Its distribution kernel is given by the formula:
	\beq\label{DF-DistrKer}
	\mathfrak{K}^A[F]\,=\,\Lambda^A\,\big[\big(\bb1_{\X}\otimes\mathcal{F}_{\X^*}\big)F\big]\circ\Upsilon
	\eeq
	where $\Upsilon:\X\times\X\rightarrow\X\times\X$ is the bijective linear transformation $\Upsilon(x,y):=\big((x+y)/2,x-y\big)$ having the inverse $\Upsilon^{-1}(z,v)=(z+v/2,z-v/2)$ and Jacobian 1 and $\mathcal{F}_{\X^*}:L^1(\X^*)\rightarrow{C}(\X)$ is the inverse Fourier transform that we recall below. We notice  that we have the equality: $$\mathfrak{K}^A[F]\,=\,\Lambda^A\,\mathfrak{K}^0[F] \mbox{ (with pointwise multiplication) } $$ and we shall simply denote by $\mathfrak{K}[F]\equiv\mathfrak{K}^0[F]$.

\paragraph{The Fourier transform.}  We use the following conventions, recalling our previous remark concerning the $2\pi$ factor in the duality map:
\begin{align}
&\mathcal{F}_{\X}:\ \mathscr{S}(\X)\rightarrow\mathscr{S}(\X^*),\quad\big(\mathcal{F}_{\X}\varphi\big)(\xi)\,:=\,\int_{\X}dx\,e^{-i<\xi,x>}\,\varphi(x)\,,\forall\varphi\in\mathscr{S}(\X)\,;\\
&\mathcal{F}_{\X^*}:\ \mathscr{S}(\X^*)\rightarrow\mathscr{S}(\X),\quad\big(\mathcal{F}_{\X^*}\varphi\big)(x)\,:=\,\int_{\X^*}d\xi\,e^{i<\xi,x>}\,\varphi(\xi)\,,\forall\varphi\in\mathscr{S}(\X^*)\,;\\ \label{DF-Ftr-Gamma}
&\mathcal{F}_{\Gamma}:\ \ell^1(\Gamma)\rightarrow{C}(\T_*),\quad\big(\mathcal{F}_{\Gamma}\vec{f}\big)(\bz^*)\,:=\,\underset{\gamma\in\Gamma}{\sum}e^{-i<\s_*(\bz^*),\gamma>}\,\vec{f}_{\gamma}\,, \forall\vec{f}\in\ell^1(\Gamma)\,;\\ \label{DF-Ftr-Tstar}
&\mathcal{F}_{\T_*}:\ C^\infty(\T_*)\rightarrow\ell^1(\Gamma),\quad\big(\mathcal{F}_{\T_*}\varphi\big)_{\gamma}\,:=\,\int_{\T_*}d\bz^*\,e^{i<\s_*(\bz^*),\gamma>}\,\varphi(\bz^*)\,, \forall\varphi\in{C}^\infty(\T_*)\,.
\end{align}

\paragraph{The H\"{o}rmander type symbols.} We shall mainly work with smooth tempered distributions on $\Xi$ that are \textit{H\"{o}rmander symbols of class} $S^p_1(\Xi)$ for some $p\in\R$, i.e. functions $F\in{C}^\infty(\Xi)$ verifying the estimations:
\beq\begin{split}
	\forall(n,m)\in\mathbb{N}\times\mathbb{N},\ &\text{ there exists } C_{n,m}>0,\ \text{such that:}\\
	&\underset{a\in\mathbb{N}^d,|a|\leq{n}}{\max}\ \underset{b\in\mathbb{N}^d,|b|\leq{m}}{\max}\ \underset{(x,\xi)\in\Xi}{\sup}\ <\xi>^{-p+m}\big|\big(\partial_x^a\partial_{\xi}^bF\big)(x,\xi)\big|\,\leq\,C_{n,m}\,.
\end{split}\eeq
The infimum of the possible constants $C_{n,m}>0$ for a given pair $(n,m)$ define a countable family of seminorms and thus a Fr\'{e}chet topology on $S^p_1(\Xi)$. \\
We say that a symbol $F\in{S}^p_1(\Xi)$ with $p>0$ \textit{is elliptic} when there exist $C>0$ and $R>0$ such that $\big|F(x,\xi)\big|\geq{C^{-1}}<\xi>^p$ for any $(x,\xi)\in\Xi$ with $|\xi|\geq{R}$. \\
For any $s\in\R$ we introduce the symbol $\mathfrak{m}_s(x,\xi)=<\xi>^s$, elliptic for $s>0$.

Applying Lemma 1.3.5 and Proposition 1.3.6 from \cite{ABG} to Formula \eqref{DF-DistrKer} we obtain the following statement:

\noindent\textbf{Theorem A.} \textit{For $B\in\Fb^2(\X)$ a closed 2-form, $A\in\Fp^1(\X)$ with $B=dA$, and  $F\in{S}^p_1(\Xi)$ with  $p\in\R$,  the operator $\Op^A(F)$ 
 leaves $\mathscr{S}(\X)$ invariant and its distribution kernel $\mathfrak{K}^A(F)\in\mathscr{S}^\prime(\X\times\X)$ has singular support in $\Delta_\X:=\{(x,x)\in\X\times\X,\ x\in\X\}$ and rapid off-diagonal decay. Moreover, for $p<0$ the kernel has integrable modulus as function of the variable $x-y\in\X$ orthogonal to $\Delta_{\X}$.}

We also recall Theorem 2.7 in \cite{IMP-3}.

\noindent\textbf{Theorem B.} \textit{If $B\in\Fb^2(\X)$ is a closed 2-form, $A\in\Fp^1(\X)$ with $B=dA$, and  $F\in{S}^p_1(\Xi)$ with $p>0$ is positive and elliptic, then $\Op^A(F)$ is essentially self-adjoint on $\mathscr{S}(\X)$  and the domain of its  closure is:
\beq
\mathcal{D}\big(\overline{\Op^A(F)}\big)=\big\{f\in{L}^2(\X),\ \Op^A(\mathfrak{m}_p)f\in{L}^2(\X)\big\}=:\mathscr{H}^p_A(\X).
\eeq}

\subsection{A brief summary of our main results}
	
Our main interest in this paper is to study the structure of the fibre operators of periodic magnetic pseudo-differential operators in their decomposition in the direct integral defined by the Bloch-Floquet transformation. 
\begin{itemize}
\item An essential starting point of our analysis is Theorem  \ref{L-Abounded} which shows that for periodic magnetic fields admitting {\it periodic vector potentials}, the magnetic symmetric quantization as developed in \cite{MP-1,IMP-1}-\cite{IMP-3} and the usual Weyl quantization (see \cite{Ho-3}) are equivalent. We do this by constructing the ``usual" Weyl symbol starting from the ``magnetic" one. 
In other words, the results proved in Theorem \ref{L-Abounded} and Proposition \ref{C-L} show that for such ``zero mean flux" periodic magnetic fields it is enough to study the Bloch-Floquet theory for periodic Weyl operators. This does not apply to magnetic fields with an ``incommensurate" non-zero mean flux. As a large number of applications deal with superpositions of different magnetic fields, a unified treatment is necessary and  Proposition \ref{C-L} opens the way to such procedures. 

\item In Section \ref{S-2} we present the extension to pseudo-differential operators of the Bloch-Floquet theory, well-known in the case of periodic differential operators, emphasizing the case of unbounded elliptic H\"{o}rmander type symbols (see Propositions \ref{P-Fzst-inv}, \ref{P-PNf-2} and \ref{P-PNf-3} and Remark \ref{R-PNf-1}).

\item Theorem \ref{T-main} is also one of our main results. There we obtain explicit formulas for the distribution kernel of the fibre operators as operators on the $d$-dimensional torus, and for the infinite matrices associated with  these operators on $\ell^2(\Gamma_*)$ via an inverse discrete Fourier transform. Using the distribution kernels we prove that the fibre operators are toroidal pseudo-differential operators in the sense of \cite{RT-06, RT-10, McL-91} and are exactly the pseudo-differential operators on the torus obtained by replacing the differential canonical operators on $\Rd$ by the invariant differential operators on the torus as in \eqref{DF-T-deriv}. 

\item Another main objective of this paper is to construct a \textit{symmetric Weyl calculus}, similar to the usual Weyl calculus on $\Rd$, associated to the $d$-dimensional torus as abelian Lie group, (case not covered by the results in \cite{W}) and prove that in the Bloch-Floquet representation the fibres of any periodic Weyl pseudo-differential operator are such symmetric toroidal operators. This is done in Section \ref{S-3}. In fact, as we conclude in our final Remark \ref{R-Final}, we may view the Bloch-Floquet representations of these periodic Weyl operators as a kind of Weyl calculus associated to the principal bundle $\X\repi\T$. This is the ``exact" (unperturbed) version of the functional calculus associated to what the authors call ``building block kinetic operators" in \cite{dNL-11}. We also discuss the connection of this Weyl pseudo-differential calculus on the $d$-dimensional torus with the usual approaches appearing in the literature. 
\end{itemize}

{
Let us end this brief description of our results by emphasizing that one may  extend our analysis to magnetic fields which, instead of obeying a zero flux condition as in \eqref{H-Bper},  are allowed to have a rational flux through the unit cell $\E\subset\X$. Moreover, see for example \cite{FT}, some more general types of magnetic fields may be considered by replacing the usual invariant differential operators on the $d$-dimensional torus with some non-trivial connection.}

\section{The Bloch-Floquet decomposition of periodic magnetic pseudo-differential operators.}\label{S-2}

\subsection{Periodic magnetic pseudo-differential operators}\label{SS-2-1}

\begin{definition}
For  $p\in\R$, we denote by ${S}^p_1(\Xi)_{\Gamma}$ the functions $F\in{S}^p_1(\Xi)$ that are $\Gamma$-periodic in the first variable, i.e. $(\tau_\gamma\otimes\bb1_{\X^*})F=F$ for any $\gamma\in\Gamma$. 
\end{definition}
The typical example is the symbol of the Schr\"odinger operator $F_0(x,\xi)=(2\pi)^{2}\xi^2+V(x)$, where  $V\in{BC}^\infty(\X)$  is $\Gamma$-periodic.

\noindent\textbf{The periodic-0-mean-face magnetic fields.}
When working with $\Gamma$-periodic $p$-forms we have to take into account that the topology of the $d$-dimensional torus is not trivial. We shall call \textit{periodic-0-mean-face magnetic field}, a periodic smooth magnetic field that  has  the following property:
\beq\label{H-Bper}
\text{For } 1\leq j,k\leq d:\quad\int\limits_{\text{\rm R}_{jk}}B=0,\quad\text{where}\quad\text{\rm R}_{jk}:=\big\{s\mathfrak{e}_j+t\mathfrak{e}_k,(s,t)\in[0,1]^2\big\}.
\eeq
In \cite{HH} the authors present an elegant cohomological argument proving that the above hypothesis is a necessary and sufficient condition for the magnetic field to be an exact 2-form on the $d$-dimensional torus and thus for the existence of a smooth $\Gamma$-periodic vector potential $A$, such that $B=dA$. We notice that in this situation $A\in\Fb^1(\X)$ and we have the following proposition.

\begin{theorem}\label{L-Abounded}
	Suppose given a symbol $F\in{S}^p_1(\Xi)$ for some $p\in\R$ and a regular magnetic field $B=dA$ with $A\in\Fb^1(\X)$. Then $\Op^A(F)$ also has a ``standard" Weyl symbol $F^A_W\in{S}^p_1(\Xi)$ such that $\Op^A(F)=\Op(F^A_W)$. Moreover we have that:
	\begin{equation} \label{eq:2.2} F^A_W(x,\xi)-F\big (x,\xi-A(x)\big )\in {S}^{p-2}_1(\Xi)
	\end{equation}
	and thus 
	\beq \label{eq:2.3}  \mathscr{H}^p_A(\X)=\mathscr{H}^p(\X)\,.
	\eeq 
\end{theorem}
\begin{proof}
	From \eqref{D-OpA} and \eqref{DF-DistrKer} the integral kernel of  $\Op^A(F)$ equals 
	$$e^{i\int_0^1 \langle\,A(x'+t(x-x'))\,,\,(x-x')\rangle}\int_{\X^*} d\xi\,  e^{i<\xi,(x-x')>} F\big((x+x')/2,\xi\big) \,,$$
	or in the variables $u=(x+x')/2$ and $v=x-x'$:
	$$K(u,v)=e^{i\int_0^1 ds\, \langle{A} (u+(s-1/2)v)\,,\, v\rangle}\int_{\X^*} d\xi \,e^{i<\xi,v>} F(u,\xi). $$
	Let $\mI^A(u,v):=\int_0^1 ds\, A\big (u+(s-1/2)v\big )$ and notice that $\mI^A(u,0)=A(u)$. Then 
	$$K(u,v)=\int_{\X^*} d\xi\, e^{i<\xi,v>} F\big (\xi -\mI^A(u,v),u\big )$$
	that corresponds to a ``standard" Weyl symbol:
	\begin{align*}
		&F_W^A(u,\eta)=\int_\X dv\,  e^{-i<\eta\,v>} \int_{\X^*} d\xi \,  e^{i<\xi,v>} F\big (u,\xi -\mI^A(u,v)\big )\\
		&=F(u,\eta-A(u))+\int_\X dv\,  e^{-i<\eta,v>} \int_{\X^*} d\xi \,  e^{i<\xi,v>} \Big (F\big (u,\xi -\mI^A(u,v)\big ) -F\big (u,\xi-\mI^A(u,0)\big )\Big ).
	\end{align*}
	For $t\in \R$ let us define 
	$G(t):=F\big (u,\xi -\mI^A(u,tv)\big )$ and make a Taylor expansion  near $t=0$:
	\begin{align*}
		G(1)-G(0)&=\sum_{j=1}^n(j!)^{-1} \frac{d^j}{dt^j}F\big (u,\xi -\mI^A(u,tv)\big )|_{t=0}+(n!)^{-1}\int_0^1 dt\,(1-t)^n  \frac{d^n}{dt^n}F\big (u,\xi -\mI^A(u,tv)\big ).
	\end{align*}
	The term $\frac{d^j}{dt^j}F\big (u,\xi -\mI^A(u,tv)\big )|_{t=0}$ generates terms which contain at least one partial derivative of $F$ with respect to $\xi$ and homogeneous polynomial factors of degree  $j$ in $v$. By partial integration in $\xi$, these polynomial factors in $v$ are transformed in partial derivatives 
	with respect to $\xi$ acting on $F$ and taking into account that $A\in\Fb^1(\X)$ has components of class $BC^\infty(\X)$, they  produce symbols belonging (at least) to the class $S_1^{p-j-1}(\Xi)$. Now consider the remainder (without the irrelevant constant factors):
	$$\tilde{F}^A_n(u,\eta):=\int_\X dv\,  e^{-i<\eta,v>} \int_{\X^*} d\xi \,  e^{i<\xi,v>} \int_0^1 dt\,(1-t)^n  \frac{d^n}{dt^n}F\big (u,\xi -\mI^A(u,tv)\big ).$$
	The derivative $\frac{d^n}{dt^n}F\big (u,\xi -\mI^A(u,tv)\big )$ generates  homogeneous polynomials of degree $n$ in $v$, which by partial integration in $\xi$ are turned into a decay of order $<\xi>^{-n+p}$. Taking partial derivatives of $\tilde{F}^A_n$ with respect to $\eta$ produces even more $v$'s which only improve the decay in $\xi$. By multiplying $\tilde{F}^A_n$ with factors $\eta_k^\alpha$ and after partial integration with respect to $v$ we may generate growing factors like $\xi_k^\alpha$, which can be controlled by choosing $n$ large enough. In any case, every derivative with respect to $\eta$ improves the decay in $\xi$ with ``one unit", while multiplication with a factor $\eta_k$ may reduce the decay in $\xi$ with ``one unit". This shows that given any $M>0$, by choosing $n$ large enough we have $\tilde{F}_n\in S_1^{-M}(\Xi)$. In fact, the derivatives with respect to $u\in\X$ are easily controlled taking into account that $F$ is a symbol of class $S^p_1(\Xi)$ and $A\in\Fb^1(\X)$. \\
	To prove \eqref{eq:2.3} it remains to prove that
	\beq \label{eq:2.4}
\big\{f\in{L}^2(\X),\ \Op(\mathfrak{m}_p)f\in{L}^2(\X)\big\}=\big\{f\in{L}^2(\X),\ \Op\big([\mathfrak{m}_p]^A_W\big)f\in{L}^2(\X)\big\}\,.
\eeq
By \eqref{eq:2.2}, we have 
\beq
[\mathfrak{m}_p]^A_W(x,\xi)=\mathfrak{m}_p\big(x,\xi-A(x)\big)\,+ r_{p-2}(x,\xi)\,,
\eeq
and we now observe that
\beq
\mathfrak{m}_p\big(x,\xi-A(x)\big)= \mathfrak{m}_p(x,\xi) +\check r_{p-1}(x,\xi)\,,
\eeq
with, noticing the boundness of $A$, 
\beq
\check{r}_{p-1}(x,\xi)=\int_0^1ds\,<\big(\nabla_\xi\mathfrak{m}_p\big)\big(x,\xi-sA(x)\,,\,A(x)>\;  \in {S}^{p-1}_1(\Xi) \,.
\eeq
Hence
$$
[\mathfrak{m}_p]^A_W(x,\xi)-\mathfrak{m}_p\big(x,\xi \big)\in {S}^{p-1}_1(\Xi)\,,
$$
defining by quantization a bounded operator $\mathscr{H}^p(\X)\rightarrow{L}^2(\X)$. Using this result and the parametrices of the quantization of the symbols $\mathfrak{m}_p$ and $[\mathfrak{m}_p]^A_W$ of order $p>0$ allows us to prove the relative boundedness of the two operators defining the two spaces in \eqref{eq:2.4}.
\end{proof}
From here on we shall assume that:
\begin{hypothesis}\label{H-main} \textit{The regular magnetic field $B\in\Fb^2(\X)$ is a periodic 0-mean-face field, i.e. satisfying \eqref{H-Bper}. 
	Under this assumption, we denote by  $A\in\Fb^1(\X)$  a periodic vector potential such that $B=dA$.}
\end{hypothesis}

\begin{remark}\label{R-OAF}
If $F\in{S}^p_1(\Xi)_\Gamma$ for some $p\in\R$ and  Hypothesis \ref{H-main} holds true,  Theorem B implies that $\mO^A(F):=\overline{\Op^A(F)}$ is a lower semi-bounded self-adjoint operator in $L^2(\X)$ that commutes with all the translations $U(\gamma):=W(\gamma,0)$ for $\gamma\in\Gamma$.
\end{remark}

If we suppose that a second regular magnetic field $\tB=d\tA\in\Fb^2(\X)$ with $\tA\in\Fp^1(\X)$ is present, we may consider the operator $\tmO(F):=\overline{\Op^{A+\tA}(F)}:\mathscr{H}^p_{\tA}(\X)\rightarrow{L}^2(\X)$ that due to Theorem B will be a lower semi-bounded self-adjoint operator.

\begin{proposition}\label{C-L}
	Suppose given a symbol $F\in{S}^p_1(\Xi)_{\Gamma}$ for some $p\in\R$ and two magnetic fields $B=dA$ periodic, satisfying Hypothesis \ref{H-main} and $\tB=d\tA$ in $\Fb^2(\X)$ with vector potentials $A\in\Fb^1(\X)$ periodic and $\tA\in\Fp^1(\X)$. Then:
	\[
	\Op^{A+\tA}(F)\,=\,\Op^{\tA}\big(F^{A}_W\big)
	\]
	with: $$F^{A}_W\,=\,
	(\bb1\otimes\mathcal{F}_{\X})\big(\mathfrak{K}^{A}[F]\circ\Upsilon^{-1}\big)\,=\,
	(\bb1\otimes\mathcal{F}_{\X})\big((\Lambda^{A}\mathfrak{K}[F])\circ\Upsilon^{-1}\big)\in{S}^p_1(\Xi)\,.$$
\end{proposition}
\begin{proof}
We have
	\begin{align*}
		&\Op^{A+\tA}(F)=\Int\,\mathfrak{K}^{A+\tA}[F],\\
		&\mathfrak{K}^{A+\tA}[F]=\Lambda^{\tA+A}\mathfrak{K}[F]=\Lambda^{\tA}\Lambda^{A}\mathfrak{K}[F]=\Lambda^{\tA}\mathfrak{K}^{A}[F],\\
		&\Int\,\Lambda^{\tA}\mathfrak{K}^{A}[F]=:\Op^{\tA}\big(F^{A}_W\big)\,,
	\end{align*}
	and using \eqref{DF-DistrKer} one obtains the desired formula for $F^{A}_W\in\mathscr{S}^\prime(\Xi)$. The fact that $F^{A}_W\in{S}^p_1(\Xi)$ follows from Theorem \ref{L-Abounded} using also that $A\in\Fb^1(\X)$.
\end{proof}

\begin{conclusion}
\textit{Having in view the results of this paragraph, from now on we deal only with periodic Weyl pseudo-differential operators $\Op(F)$ with $F\in{S}^p_1(\Xi)_{\Gamma}$ with $p\in\R$, and cover also the magnetic quantization of these symbols in a periodic magnetic field satisfying Hypothesis \eqref{H-main}.}
\end{conclusion}

Finally let us consider a linear operator $T\in\mathcal{L}\big(\mathscr{S}(\X);\mathscr{S}^\prime(\X)\big)$ that commutes with all the unitary translations with lattice elements $\gamma\in\Gamma$. Due to the Schwartz Kernels Theorem (\cite{Schw}), there exists a distribution $\mathfrak{K}_T\in\mathscr{S}^\prime(\X\times\X)$ such that: $\big(\psi\,,\,T\phi\big)_{L^2(\X)}=\big\langle\mathfrak{K}_T\,,\,\overline{\psi}\otimes\phi\big\rangle_{\mathscr{S}(\X\times\X)}$. The commutation property with the translations with elements from $\Gamma$ clearly imply the equalities:
\beq\label{F-com-KAF-gamma}
(\tau_{\gamma}\otimes\tau_{\gamma})\mathfrak{K}_T=\mathfrak{K}_T,\quad\forall\gamma\in\Gamma
\eeq
and we notice that the oscillating factor $\Lambda^A$ with $\Gamma$-periodic vector potential $A\in\Fb^1(\X)$ satisfies similar identities: 
\beq
\Lambda^A(x+\gamma,y+\gamma)\,=\,\Lambda^A(x,y),\quad\forall\gamma\in\Gamma.
\eeq
Moreover, we notice that formula \eqref{DF-DistrKer} of the distribution kernel of a Weyl quantized operator has the inverse:
\beq
F=\big(\bb1_{\X}\otimes\mathcal{F}_{\X}\big)\big[[\Lambda^A]^{-1}\mathfrak{K}^A[F]\big]\circ\Upsilon^{-1}
\eeq
and we conclude that given $T\in\mathcal{L}\big(\mathscr{S}(\X);\mathscr{S}^\prime(\X)\big)$ that commutes with all the unitary translations with lattice elements $\gamma\in\Gamma$, for any $\Gamma$-periodic vector potential $A\in\Fb^1(\X)$ there exists a unique $\Gamma$-periodic symbol $F^A_T\in\mathscr{S}^\prime(\Xi)_{\Gamma}$ such that:
\beq
F^A_T=\big(\bb1_{\X}\otimes\mathcal{F}_{\X}\big)\big[[\Lambda^A]^{-1}\mathfrak{K}_T\big]\circ\Upsilon^{-1}
\eeq 
i.e. $T=\Op^A(F^A_T)$.

\subsection{Two unitary transformations}

Let us recall very briefly the two unitary transformations allowing to study the structure of the $\Gamma$-periodic operators in $L^2(\X)$.

\paragraph{The Bloch-Floquet (BF) transformation.} For $\bz^*\in\T_*$ we define the complex space:
\beq\label{DF-Fzst}
\mathscr{F}_{\bz^*}\,:=\,\big\{\hat{f}\in{L}^2_{\text{\tt loc}}(\X),\ \tau_{\gamma}f=\chi_{\bz^*}(\gamma)\,f,\ \forall\gamma\in\Gamma\big\}
\eeq
and endow it with the scalar product:
\beq\label{DF-psc-F}
\big(\hat{f}\,,\,\hat{g}\big)_{\mathscr{F}_{\bz^*}}\,:=\,\int_{\E}d\hat{x}\,\overline{\hat{f}(\hat{x})}\,\hat{g}(\hat{x}).
\eeq
that makes it a Hilbert space and define the direct integral $\mathscr{F}\ :=\ \int_{{\T}_*}^{\oplus}\hspace*{-2pt}d\bz^*\,\mathscr{F}_{\bz^*}$ in the sense of \cite{Dix}.
One easily verifies that the following map defines a unitary operator $L^2(\X)\overset{\sim}{\rightarrow}\mathscr{F}$:
\beq\label{DF-BFtr}
\big(\U_{BF}f\big)(\bz^*,x)\ :=\,\underset{\gamma\in\Gamma}{\sum}\,e^{-i<\s_*(\bz^*),\gamma>}\,f(x+\gamma)
\eeq
having the inverse:
\beq\label{D-BFtr-inv}
\big(\U_{BF}^{-1}\hat{f}\big)(\hat{x}+\alpha)\,=\,\int_{\T_*}\hspace*{-4pt}d\bz^*\,e^{i<\s_*(\bz^*),\alpha>}\,\hat{f}(\bz^*,\hat{x}).
\eeq
\begin{proposition}\label{P-UBF-S}
	$\U_{BF}[\mathscr{S}(\X)]=\Big[\int_{{\T}_*}^{\oplus}\hspace*{-2pt}d\bz^*\,\big[\mathscr{F}_{\bz^*}\bigcap{BC}^\infty(\X)\big]\Big]\,\bigcap\,C^\infty(\T_*)$.
\end{proposition}
\begin{proof}
	For $\varphi\in\mathscr{S}(\X)$, Definition \eqref{DF-BFtr} implies:
	\beq
	\big(\U_{BF}\varphi\big)(\bz^*,x)=\underset{\gamma\in\Gamma}{\sum}\,e^{-i<\s_*(\bz^*),\gamma>}\,\varphi(x+\gamma)
	\eeq
	and one verifies immediately that $\U_{BF}\varphi\in\Big[\int_{{\T}_*}^{\oplus}\hspace*{-2pt}\big[\mathscr{F}_{\bz^*}\bigcap{BC}^\infty(\X)\big]\Big]\,\bigcap\,C^\infty(\T_*)$.
	
	Reciprocally, let us take some $\hat{\varphi}\in\Big[\int_{{\T}_*}^{\oplus}\hspace*{-2pt}d\bz^*\,\big[\mathscr{F}_{\bz^*}\bigcap{BC}^\infty(\X)\big]\Big]\,\bigcap\,C^\infty(\T_*)$ and use \eqref{D-BFtr-inv} to write:
	\beq
	\hat{\varphi}=\U_{BF}\varphi,\quad\varphi(\hat{x}+\alpha)=\int_{\T_*}\hspace*{-4pt}d\bz^*\,e^{i<\s_*(\bz^*),\alpha>}\,\hat{\varphi}(\bz^*,\hat{x}).
	\eeq
\end{proof}

\begin{notation}
	We shall use the notations: $\mathscr{F}_{\bz^*}^\infty:=\mathscr{F}_{\bz^*}\bigcap{C}^\infty(\X)$ and $\mathscr{F}^\infty:=\int_{\T_*}^{\oplus}d\bz^*\,\mathscr{F}^\infty_{\bz^*}$.
\end{notation}

\paragraph{The Bloch-Floquet-Zak (BFZ) transformation.}

We identify any $\Gamma$-periodic function with a function on the $d$-dimensional torus defined by composing this last one with the canonical projection on the quotient.
 For any $\bz^*\in\T_*$ let us define the complex space:
\beq\label{D-BFZ-Gsp}
\mathscr{G}\,:=\,\big\{\tf\in{L}^2_{\text{\tt loc}}\big(\X^*;L^2(\T)\big)),\ (\tau_{\gamma^*}\otimes\bb1_{\T})\tf=(1\otimes\hat{\chi}_{\gamma^*})\tf,\ \forall\gamma^*\in\Gamma\big\}
\eeq
and endow it with the scalar product:
\beq\label{DF-psc-G}
\big(\tf\,,\,\tg\big)_{\mathscr{G}}\,:=\,\int_{\E_*}\hspace*{-4pt}d\hat{\xi}\,\int_{\T}d\bz\,\overline{\tf(\hat{\xi},\bz)}\,\tg(\hat{\xi},\bz)
\eeq
making it a Hilbert space.
Then one easily verifies that the following map defines a unitary operator $L^2(\X)\overset{\sim}{\rightarrow}\mathscr{G}$:
\beq\label{DF-BFZ-trsf}
\big(\U_{BFZ}f\big)(\xi,\bz)\ :=\,\underset{\gamma\in\Gamma}{\sum}\,e^{-i<\xi,\s(\bz)+\gamma>}\,f(\s(\bz)+\gamma)
\eeq
having the inverse:
\beq
\big(\U_{BFZ}^{-1}\hat{f}\big)(\hat{x}+\alpha)\,=\,\int_{\T_*}\hspace*{-4pt}d\bz^*\,e^{i<\s_*(\bz^*),\alpha+\hat{x}>}\,\hat{f}(\bz^*,\hat{x}).
\eeq

\begin{remark}\label{R-BF-BFZ}
If we define the smooth function $\mathfrak f$ by
\beq\label{eq:2.17} \X^*\times\X\ni(\xi,x)\mapsto\mathfrak{f}(\xi,x):={e}^{-i<\xi,x>}\in\mathbb{S}\,,\eeq
 we notice that $\mathfrak{f}(\xi,x)=\chi_{\xi}(x)$ and:
\beq\label{eq:2.18}
\forall{f}\in{L}^2(\X):\qquad\big(\U_{BFZ}f\big)(\xi,\bz)\,=\,\mathfrak{f}\big(\xi,\s(\bz)\big)\,\big(\mathfrak\U_{BF}f\big)\big(\p_*(\xi),\s(\bz)\big).
\eeq
It follows that if an operator $T$ has a decomposable image $\int_{\T_*}^{\oplus}d\bz\,\hat{T}_{\bz^*}$ under the BF transformation with respect to the direct integral decomposition, then its image under the BFZ transformation is also decomposable as a family $\{\widetilde{T}_{\xi}\}_{\xi\in\X^*}$ of operators acting in $L^2(\T)$ with 
$$\widetilde{T}_{\xi}=e^{-i<\xi,\cdot>}\hat{T}_{\p_*(\xi)}e^{i<\xi,\cdot>}\,.$$
\end{remark}

\begin{remark}\label{R-F-vect-bdl}
One easily notice that the BFZ-representation is the $\Gamma_*$-equivariant representation of $\X^*\cong\Rd$ induced by the unitary representation:
\beq\label{DF-hatUcirc}
\hat{U}^\circ:\Gamma_*\rightarrow\mathbb{U}\big(L^2(\T)\big),\quad\big(\hat{U}^\circ(\gamma^*)f\big)(\omega):=e^{-i<\gamma^*,\omega>}f(\omega),\,\forall\omega\in\T.
\eeq
From the general abstract theory one knows that this is unitarily equivalent to the representation on sections of the Euclidean vector bundle $\mathfrak{F}\repi\T_*$  associated with  the principal bundle $\X^*\repi\T_*$ with fibre $\Gamma_*$ and the representation \eqref{DF-hatUcirc} of $\Gamma_*$. 
\end{remark}

\subsection{The Bloch-transformation of $\Op(F)$}\label{SSS-Btrs-OAF}

\subsubsection{The fibre operators}

Let us compute the image of $\Op(F)$ through the BF transformation.
If $\hat{\varphi}\in\mathscr{F}^\infty\bigcap\,C^\infty(\T_*)$ we obtain:
\begin{align}\label{F-PN-N-2}
&\big[\big(\U_{BF}\Op(F)\U_{BF}^{-1}\big)\hat{\varphi}\big](\bz^*,x)=\underset{\gamma\in\Gamma}{\sum}e^{-i<\s_*(\bz^*),\gamma>}\big[\big(\Op(F)\U_{BF}^{-1}\big)\hat{\varphi}\big](x+\gamma)\\ \nonumber
&\quad=\underset{\gamma\in\Gamma}{\sum}e^{-i<\s_*(\bz^*),\gamma>}\big[U(\gamma)\big(\Op(F)\U_{BF}^{-1}\big)\hat{\varphi}\big](x)=\underset{\gamma\in\Gamma}{\sum}e^{-i<\s_*(\bz^*),\gamma>}\big[\big(\Op(F)U(\gamma)\U_{BF}^{-1}\big)\hat{\varphi}\big](x)
\end{align}
where the series in $\gamma\in\Gamma$ is convergent due to the rapid decay of $\big(\Op(F)\U_{BF}^{-1}\big)\hat{\varphi}\in\mathscr{S}(\X)$ (as implied by Theorem A and Proposition \ref{P-UBF-S}). The quasi-periodicity of the functions in $\mathscr{F}^\infty_{\tbz^*}$ implies that $,e^{i<\s_*(\tbz^*),\iota(y)>}\,\hat{\phi}(\tbz^*,\hat{y})=\hat{\phi}(\tbz^*,y)$ and we obtain:
\begin{align}\label{F-PN-N-1}
&\big[\big(\Op(F)U(\gamma)\U_{BF}^{-1}\big)\hat{\varphi}\big](x)=\\ \nonumber
&\hspace*{0.5cm}=\int_{\X}dy\int_{\X^*}d\eta\,e^{i<\eta,x-y>}\,F\big((x+y)/2,\eta\big)\int_{\T_*}d\tbz^*\,e^{i<\s_*(\tbz^*),\gamma>}\,\hat{\phi}(\tbz^*,y).
\end{align}

We work with the extension of $\Op(F)$ to $BC^\infty(\X)\subset\mathscr{S}^\prime(\X)$ weakly in distribution sense and Theorem A implies that it preserves smoothness.

\begin{proposition}\label{P-Fzst-inv}
Suppose $F\in{S}^p_1(\Xi)_{\Gamma}$. Then the canonical extension of $\Op(F)$ to $\mathscr{S}^\prime(\X)$ leaves invariant the subspaces $\mathscr{F}^\infty_{\bz^*}$ for any $\bz^*\in\T_*$.
\end{proposition}
\begin{proof}
Suppose fixed some $\bz^*\in\T_*$ and some $\hat{\phi}\in\mathscr{F}_{\bz^*}^\infty$ and let us fix a $\Gamma$-partition of unity with cut-off function $\rho\in{C}^\infty_0(\X)$, as in Definition \ref{D-G-psrtunit}. Then  let us translate it with any $\alpha\in\Gamma$ and fix some test function $\psi\in\mathscr{S}(\X)$ in order to compute:
\begin{align*}
	&\big\langle\tau_{\alpha}\Op(F)\hat{\phi}\,,\,\psi\big\rangle_{\mathscr{S}(\X)}\hspace*{-0.2cm}=\Big\langle\big[\underset{\gamma\in\Gamma}{\sum}\big(\tau_{\alpha}\Int\,\mathfrak{K}[F](\rho\circ\tau_\gamma)\hat{\phi}\big)\Big]\,,\,\psi\Big\rangle_{\mathscr{S}(\X)}\hspace*{-0.3cm}\\
	&\quad=\underset{\gamma\in\Gamma}{\sum}\big\langle(\tau_{\alpha}\otimes\bb1)\mathfrak{K}[F]\,,\,\psi\otimes\big[(\rho\circ\tau_\gamma)\hat{\phi}\big]\big\rangle_{\mathscr{S}(\X\times\X)}\hspace*{-0.2cm}=\underset{\gamma\in\Gamma}{\sum}\big\langle(\bb1\otimes\tau_{-\alpha})\mathfrak{K}[F]\,,\,\psi\otimes\big[(\rho\circ\tau_\gamma)\hat{\phi}\big]\big\rangle_{\mathscr{S}(\X\times\X)}\\ \numberthis\label{F-PN-N-3}
	&=\big\langle\Op(F)(\tau_{\alpha}\hat{\phi})\,,\,\psi\big\rangle_{\mathscr{S}(\X)}=\big\langle\hat{\chi}_{\alpha}(\bz^*)\Op(F)\hat{\phi}\,,\,\psi\big\rangle_{\mathscr{S}(\X)}\,.  
\end{align*}
\end{proof}

\begin{notation}\label{N-Fzst-inv}
Let $\widehat{\Op}(F)_{\bz^*}$ be the extension of $\Op(F)$ to $\mathscr{F}_{\bz^*}$ using Proposition \ref{P-Fzst-inv}.
\end{notation}
Finally \eqref{F-PN-N-1}, \eqref{F-PN-N-2}, Proposition \ref{P-Fzst-inv} and the compactness of $\T_*$ imply:
\beq\label{DF-OpABF}
\U_{BF}\Op(F)\U_{BF}^{-1}\,=\,\int_{{\T}_*}^{\oplus}\hspace*{-4pt}d\bz^*\,\widehat{\Op}(F)_{\bz^*}.
\eeq

Fixing $\bz^*=\mathfrak{i}^*$ (as in \eqref{DF-1}), $\xi\in\X^*$ and $\Phi\in\mathscr{F}_{\mathfrak{i}^*}\bigcap{BC}^\infty(\X)$, the Remark \ref{R-BF-BFZ} implies:
\begin{align}\label{DF-OpABFZ}
	\big(e^{-i<\xi,\cdot>}\widehat{\Op}(F)_{\p_*(\xi)}e^{i<\xi,\cdot>}\Phi\big)(x)&=\int_{\X}dy\int_{\X^*}\hspace*{-4pt}d\eta\,e^{-i<\xi,(x-y)>}\,e^{i<\eta,(x-y)>}\,F\big((x+y)/2,\eta\big)\Phi(y)\nonumber \\
	&=\big(\widehat{\Op}\big((\bb1_{\X}\otimes\tau_{\xi})F\big)_{\mathfrak{i}^*}\Phi\big)(x)\,.
\end{align}

\begin{remark}\label{R-1}
	Let us notice that the point $\mathfrak{i}^*$ appearing in the above construction has in fact no special role and can be replaced by any $\bz^*_\circ\in\T_*$ by simply replacing the function $\mathfrak{f}:\X^*\times\X\rightarrow\mathbb{S}$ introduced in \eqref{eq:2.17}  with the function $$(\xi,x) \mapsto \mathfrak{f}_{\bz^*_\circ}(\xi,x):=\mathfrak{f}(\xi-\s_*(\bz^*_\circ),x)\,.$$
\end{remark}

\begin{notation}\label{N-60}
	For any $F\in\mathscr{S}^\prime(\Xi)$ and any $\xi\in\X^*$ we write $ F_\xi:=(\bb1_{\X}\otimes\tau_{\xi})F$.
\end{notation}
The above arguments lead to  the following statement.
\begin{proposition}\label{P-60}
Given $p\in\R$ and  a symbol $F\in{S}^p_1(\Xi)_{\Gamma}$, the operator $\Op(F)\in\mathcal{L}\big(\mathscr{S}(\X)\big)$, is unitarily equivalent via the BFZ transformation \eqref{DF-BFZ-trsf} with a decomposable operator defined by a smooth family $\big\{\widetilde{\Op}(F)_{\xi}\big\}_{\xi\in\X^*}$ acting in $L^2(\T)$, given by:
\begin{align}\label{DF-OpA-L2T}
	\forall\xi\in\X^*,\ \forall\mathring{\varphi}\in{C}^\infty(\T),\qquad&\big(\widetilde{\Op}(F)_{\xi}\mathring{\varphi}\big)(\bz):=\big(\widehat{\Op}\big(F_{\xi}\big)_{\mathfrak{i}^*}(\mathring{\varphi}\circ\p)\big){ \big(\s(\bz)\big)}\,,
\end{align}
and having a distribution kernel that we denote by $\mathfrak{K}_{\T}[F_\xi]$. 
\end{proposition}

\begin{remark}
Putting together Proposition \ref{P-60} with our previous Remark \ref{R-F-vect-bdl}, let us denote by $\mathfrak{B\hspace*{-2pt}F}\repi\T_*$ the vector bundle associated to the principal bundle $\X^*\repi\T_*$ by the automorphism representation induced by \eqref{DF-hatUcirc} on $\mathbb{B}\big(L^2(\T)\big)$ and conclude that $\U_{BFZ}\Op(F)\U_{BFZ}^{-1}$ defines a smooth section $\T_*\ni\bz^*\mapsto\widetilde{\Op}(F)_{\s_*(\bz^*)}\in\mathfrak{B\hspace*{-2pt}F}$.
\end{remark}

\subsubsection{The distribution kernels.}

First we notice that the formula \eqref{DF-DistrKer} with $A=0$ implies the following equality with oscillating integrals for the distribution kernel of $\Op(F)$:
\beq\label{F-II-11}
\mathfrak{K}[F](x,y)=\int_{\X^*}d\eta\,e^{i<\eta,x-y>}\,F\big((x+y)/2,\eta\big).
\eeq
If $F\in{S}^p_1(\Xi)_{\Gamma}$ with $p<0$ we notice that the formula \eqref{F-com-KAF-gamma} may be written as:
\beq\label{F-II-1}
\mathfrak{K}[F](\hat{x}+\alpha,\hat{y}+\beta)=\hspace*{-0.4cm}\underset{(\alpha,\beta)\in\Gamma\times\Gamma}{\sum}\big[(\tau_{\alpha}\otimes\tau_{\alpha})\mathfrak{K}[F]\big](\hat{x},\hat{y}+\beta-\alpha)=\hspace*{-0.4cm}\underset{(\alpha,\beta)\in\Gamma\times\Gamma}{\sum}(1\otimes\tau_{\beta-\alpha})\mathfrak{K}[F](\hat{x},\hat{y}).
\eeq
In order to write a similar decomposition for the general case $p\in\R$ we shall regularize the distribution $\mathfrak{K}[F]$ using a family of cut-off functions in the variable $\eta\in\X^*$, namely:
\beq\begin{split}
	\mathfrak{K}[F]_N(x,y):=&\int_{\X^*}d\eta\,e^{i<\eta,x-y>}\,\theta_N(\eta)\,F\big((x+y)/2,\eta\big),\\
	&\theta\in{C}^\infty_0\big(\X^*;[0,1]\big),\ |\xi|\leq1/2\,\Rightarrow\,\theta(\xi)=1,\ \supp\theta\subset\{|\xi|{ \leq 1}\}\,,\\
	&\theta_N(\xi):=\theta(N^{-1}\xi),\quad\forall{N}\in\Nb\,.
\end{split}\eeq
Then, for any ${N}\in\Nb$ the kernel $\mathfrak{K}_N[F]$ is the Fourier transform of a smooth compactly supported function { in the $\eta$ variable} $F^{(N)}(z,\eta):=\theta_N(\eta)F(z,\eta)$ and thus it is a function of class $\mathscr{S}(\X\times\X)$. 
Moreover $F^{(N)}(z,\eta):=\theta_N(\eta)F(z,\eta)$ still defines a H\"{o}rmander type symbol $\Gamma$-periodic in the first variable $z\in\X$.
As $F^{(N)}$ belongs to the H\"{o}rmander class $S^{-\infty}(\Xi)_{\Gamma}\subset{S}^p_1(\Xi)_{\Gamma}$, we deduce that the proof of Proposition \ref{P-Fzst-inv} remains true for $\Op(F^{(N)})=\Op(F)_N$ and thus the extension of $\Op(F)_N$ at $\mathscr{S}^\prime(\X)$ leaves invariant any $\mathscr{F}_{\bz^*}$, for all $N\in\Nb$ defining the operators $\widehat{\Op}(F^{(N)})_{\bz^*}$. The commutation relations \eqref{F-com-KAF-gamma} remain valid for $\mathfrak{K}[F^{(N)}]$ for any $N\in\Nb$.

\begin{proposition}\label{P-II-40}
	Fixing any $\bz^*\in\T_*$, for any $N\in\Nb$, $\widehat{\Op}(F^{(N)})_{\bz^*}$ has the following structure  as operator on $\mathscr{F}^\infty_{\bz^*}\subset\mathscr{E}(\X)$:
	\beq\label{bpm1} \begin{split}
		&\widehat{\Op}(F^{(N)})_{\bz^*}\,=\,\Int_{\X}\mathfrak{K}[F^{(N)}]_{\bz^*},\\
		&\mathfrak{K}[F^{(N)}]_{\bz^*}(\hat{x},\hat{y})=\underset{\beta\in\Gamma}{\sum}\hat{\chi}_{\bz^*}(\beta)\,\mathfrak{K}[F^{(N)}](\hat{x},\hat{y}+\beta)\,.
	\end{split}\eeq
\end{proposition}
\begin{proof}
	Let us fix two functions $\hat{\phi}$ and $\hat{\psi}$ in $\mathscr{F}^\infty_{\bz^*}$ and use Proposition \ref{P-Fzst-inv} in order to compute:
	\begin{align*}
		\big(\hat{\psi}\,,\,\widehat{\Op}(F^{(N)})_{\bz^*}\hat{\phi}\big)_{\bz^*}&=\underset{M\nearrow\infty}{\lim}\int_{\E}d\hat{x}\,\overline{\hat{\psi}(\hat{x})}\int_{\X}dy\,\mathfrak{K}[F^{(N)}](\hat{x},y)\,\hat{\phi}_M(y)\\
		&\hspace*{-2.5cm}=\underset{M\nearrow\infty}{\lim}\int_{\E}d\hat{x}\,\overline{\hat{\psi}(\hat{x})}\int_{\E}d\hat{y}\hat{\phi}(\hat{y})\underset{\beta\in\Gamma}{\sum}\mathfrak{K}[F^{(N)}](\hat{x},\hat{y}+\beta)\hat{\chi}_{\bz^*}(\beta)\underset{|\gamma|\leq{M}}{\sum}\rho(\hat{y}+\beta+\gamma).
	\end{align*}
\end{proof} 
\begin{proposition}
For any $N\in\Nb$ and for $\bz^*=\mathfrak{i}^*$ the integral kernel $\mathfrak{K}[F^{(N)}]_{\mathfrak{i}^*}$ extends to a smooth $\Gamma\times\Gamma$-periodic function.
\end{proposition}
\begin{proof}
For any pair $(\hat{x},\hat{y})\in\E\times\E$, using the commutation relations \eqref{F-com-KAF-gamma} for $\mathfrak{K}[F]_N$ we have the equalities:
\begin{align*}
\underset{\gamma\in\Gamma}{\sum}\hat{\chi}_{\mathfrak{i}^*}(\gamma)\,\mathfrak{K}[F]_N(\hat{x}+\alpha,\hat{y}+\beta+\gamma)=\underset{\gamma\in\Gamma}{\sum}\mathfrak{K}[F]_N(\hat{x},\hat{y}+\beta-\alpha+\gamma)\,=\,\underset{\tgamma\in\Gamma}{\sum}\mathfrak{K}[F]_N(\hat{x},\hat{y}+\tgamma).
\end{align*}
\end{proof}
\begin{definition}
We define the following distribution kernels in $\mathscr{D}^\prime(\T\times\T)$:
\beq
\forall{N}\in\Nb:\quad\mathfrak{K}_{\T}[F^{(N)}]:=\mathfrak{K}[F^{(N)}]_{\mathfrak{i}^*}\circ(\s\otimes\s).
\eeq
\end{definition}

\begin{proposition}\label{P-II-10}
	Supose given $F\in{S}^p_1(\Xi)_{\Gamma}$ and let us fix $\hat{\phi}\in\mathscr{F}_{\bz^*}$ for some $\bz^*\in\T_*$. Then given any $\epsilon>0$, for any $\mu\in\mathbb{N}^d$ and any $n\geq(p+d+1)/2$, there exists some constant $C(n,\mu,\nu)>0$ and some $N^{\prime\prime}_\epsilon\in\Nb$ such that for any $N\geq{N}^{\prime\prime}_\epsilon$:
	\begin{align}\label{F-II-10}
	&\underset{x\in\X}{\sup}\big|\big[\partial_x^\mu\big(\widehat{\Op}(F)_{\bz^*}-\widehat{\Op}(F^{(N)})_{\bz^*}\big)\hat{\phi}\big](x)\big|\leq\\ \nonumber
	&\quad\leq\,C(n,\mu)\,\epsilon\,\Big[\underset{(x,\xi)\in\Xi}{\sup}\ \underset{|\nu|\leq|\mu|+p+d+1}{\max}\ \underset{|\nu^\prime|\leq2n}{\max}\big[\partial_x^\nu\partial_\xi^{\nu^\prime}F\big](x,\xi)\big|\Big]\Big[\underset{x\in\X}{\sup}\underset{|\nu^{\prime\prime}|\leq p+d+1}{\max}\big|\big[\partial_x^{\nu^{\prime\prime}}\hat{\phi}\big](x)\big|\Big].
\end{align}
\end{proposition}
\begin{proof} Let us write:
\begin{align}
	&\big[\big(\widehat{\Op}(F)_{\bz^*}-\widehat{\Op}(F^{(N)})_{\bz^*}\big)\hat{\phi}\big](\hat{x}+\alpha)\\ \nonumber
	&\qquad=\int_{\X}dy\Big[\int_{\X^*}d\eta\,e^{i<\eta,\hat{x}-y+\alpha>}\,\big(1-\theta_N(\eta)\big)F\big((\hat{x}+y+\alpha)/2,\eta\big)\Big]\hat{\phi}(y)
\end{align}
\begin{align*}
&\qquad=e^{-i<\s_*(\bz^*),\alpha>}\int_{\X}dy\Big[\int_{\X^*}d\eta\,e^{i<\eta,\hat{x}-(y-\alpha)>}\,\big(1-\theta_N(\eta)\big)
{F\big((\hat{x}+y-\alpha)/2+\alpha,\eta\big)}\Big]\hat{\phi}(y-\alpha)\\
&\qquad=e^{-i<\s_*(\bz^*),\alpha>}\int_{\X}dy
\Big[\int_{\X^*}d\eta\,
e^{i<\eta,\hat{x}-y>}\,\big(1-\theta_N(\eta)\big)
F\big((\hat{x}+y)/2,\eta\big)\Big]\hat{\phi}(y)\\
&\qquad=e^{-i<\s_*(\bz^*),\alpha>}\int_{\X}dy<y>^{-2n}\int_{\X^*}d\eta\,<\eta>^{-(p+d+1)}e^{i<\eta,\hat{x}-y>}\,\times\\
&\hspace*{4cm}\Big[\big(\bb1-\Delta_{y}\big)^{(p+d+1)}\Big(\big(\bb1-\Delta_{\eta}\big)^n\big(1-\theta_N(\eta)\big)F\big((\hat{x}+y)/2,\eta\big)\Big)\hat{\phi}(y)\Big]
\end{align*}
and some usual integration by parts using the oscillating factor $e^{i<\eta,x-y>}$ allows to obtain the desired estimation taking into account that:
\begin{align}
	\int_{\X^*}d\eta\,<\eta>^{-(p+d+1)}\big(1-\theta_N(\eta)\big)\,\underset{N\nearrow\infty}{\longrightarrow}\,0,\quad\big(\partial_\eta^{\kappa}\theta_N\big)(\eta)\,\underset{N\nearrow\infty}{\text{\LARGE{$\sim$}}}\,N^{-|\kappa|},\ \forall\kappa\in\mathbb{N}^d.
\end{align}
\end{proof}
\begin{corollary}\label{C-II-10}
The limit $\widehat{\Op}(F^{(N)})_{\bz^*}\,\underset{N\nearrow\infty}{\longrightarrow}\,\widehat{\Op}(F)_{\bz^*}$ exists uniformly on bounded subsets of $\mathscr{F}_{\bz^*}$ for any $\bz^*\in\T_*$. 
\end{corollary}

We are now able to formulate one of the main results of our paper:
\begin{theorem}\label{T-main} 
Given a symbol $F\in{S}^p_1(\Xi)_{\Gamma}$ for some $p\in\R$, the operator $\Op(F)\in\mathcal{L}\big(\mathscr{S}(\X)\big)$, is unitarily equivalent via the BFZ transformation \eqref{DF-BFZ-trsf} with a smooth map:
\beq
\X^*\ni\xi\,\mapsto\,\widetilde{\Op}(F)_{\xi}\in\mathcal{L}\big(C^\infty(\T)\big)
\eeq
that has the properties:
\begin{enumerate}[A. ]
\item  For any $\gamma^*\in\Gamma_*$ we have the identity: $\widetilde{\Op}(F)_{\xi+\gamma^*}\,=\,\hat{\chi}_{\gamma^*}\widetilde{\Op}(F)_{\xi}\hat{\chi}_{\gamma^*}^{-1}$.
\item 
{Let us consider the decomposition $\Gamma\ni\tgamma=2\gamma+\kappa\in[2\Gamma]+\Sigma_1$ where
\beq\label{F-II-31}
[2\Gamma]\,:=\,\underset{1\leq j\leq d}{\bigoplus}[2\Z]\e_j\,\subset\,\Gamma,\qquad\Sigma_1\,:=\,\big\{\kappa\in\Gamma,\ \kappa_j\in\{0,1\}\,\forall{j}\in\underline{d}\}\big\}\subset\Gamma,
\eeq
and similarly for the dual indices $\gamma^*\in\Gamma^*$: 
$\gamma^*=2\alpha^*+\kappa^*$ with $\alpha^*\in\Gamma_*$ and $\kappa^*\in\Sigma_1^*$. Then 
$\widetilde{\Op}(F)_{\xi}$} has a distribution kernel $\mathfrak{K}[F_\xi]\in\mathscr{D}^\prime(\T\times\T)$ given by: 
\beq
\mathfrak{K}_\T[F_\xi]=\big[\big(\bb1\otimes\mathcal{F}_{\Gamma_*}\big)\,\Phi_{\bullet}[F_{\xi}]\big]\circ(\s\otimes\s)
\eeq
where $\Phi_{\bullet}[F]$ is a uniformly polynomially bounded $\Gamma_*$-sequence of smooth functions on $\X\times\X$ defined by:
\begin{align}\label{F-53-2-St}
	&\Phi_{\alpha^*}[F^{(N)}_\xi](x,y)=\\ \nonumber
    &\hspace*{1cm}=\underset{\kappa^*\in\Sigma_1^*}{\sum}e^{(i/2)<\kappa^*,x-y>}\underset{\kappa\in\Sigma_1}{\sum}\big[\frac{e^{(i/2)<\kappa^*,\kappa>}}{2^d}F^{(N)}_\xi\big((x+y+\kappa)/2,\alpha^*+(1/2)\kappa^*\big)\big].
\end{align}
{This function is also} $\Gamma\times\Gamma$-periodic, i.e. $\Phi_{\alpha^*}[F^{(N)}_\xi](x+\alpha,y+\beta)=\Phi_{\alpha^*}[F^{(N)}_\xi](x,y)$ for any $(\alpha,\beta)\in\Gamma\times\Gamma$.
\item {Let us consider the discrete Fourier decomposition of the symbol $F\in{S}^p_1(\Xi)_{\Gamma}$ with respect to its first variable in $\X$:
\beq\label{F-61}
\widehat{F}_{\alpha^*}(\xi)\,:=\,\int_{\T}d\bz\,e^{i<\alpha^*,\s(\bz)>}\,F\big(\s(\bz),\xi\big).
\eeq Then the fibre $\widetilde{\Op}(F)_{\xi}$} is unitarily equivalent via the inverse discrete Fourier transform $\mathcal{F}_{\T}:L^2(\T)\overset{\sim}{\longrightarrow}\ell^2(\Gamma_*)$ with an operator on $\ell^2(\Gamma_*)$ having the following matrix elements with respect to the canonical orthonormal basis of $\ell^2(\Gamma_*)$:
\beq\label{F-52}
\big[\mathcal{F}_{\T}\big(\widetilde{\Op}(F)_{\xi}\big)\mathcal{F}_{\T}^{-1}\big]_{\alpha^*,\beta^*}\,=\,{\widehat{F}_{\alpha^*-\beta^*}\big(\xi -(\alpha^*+\beta^*)/2\big)}.
\eeq	
\end{enumerate}
\end{theorem}
\begin{proof}
Using Proposition \ref{P-60}, the existence of the family $\big\{\widetilde{\Op}(F)_{\xi}\big\}_{\xi\in\X^*}$ and the existence of the distribution kernels is a consequence of the Schwartz Kernel Theorem, the space $C^\infty(\T)$ being nuclear. The smoothness of this family with respect to $\xi\in\X^*$ clearly follows from \eqref{DF-OpA-L2T} in Proposition \ref{P-60}. We shall now concentrate on the properties of these distribution kernels.

Due to the Corollary \ref{C-II-10} we have the equality:
\begin{align}
&\widehat{\Op}(F_{\s_*(\bz^*)})_{\mathfrak{i}^*}\big(\U_{BF}\varphi\big)_{\mathfrak{i}^*}=\underset{N\nearrow\infty}{\lim}\Int_{\X}\,\mathfrak{K}[F^{(N)}_{\s_*(\bz^*)}]_{\mathfrak{i}^*}\big(\U_{BF}\varphi\big)_{\mathfrak{i}^*},\quad\forall{\varphi}\in\mathscr{S}(\X).
\end{align}
Thus, if we consider the identification $\mathscr{F}_{\mathfrak{i}^*}^\infty=\mathscr{E}_{\Gamma}(\X)\cong{C}^\infty(\T)$ given by $C^\infty(\T)\ni\mathring\phi\mapsto\mathring{\phi}\circ\p\in\mathscr{E}_{\Gamma}(\X)$, given two functions $\mathring{\phi}$ and $\mathring{\psi}$ in $C^\infty(\T)$ we have the equality:
\beq
\big(\mathring{\psi}\circ\p\,,\,\widehat{\Op}(F_{\hat{\xi}})_{\mathfrak{i}^*}\,(\mathring{\phi}\circ\p)\big)_{\mathscr{F}_{\mathfrak{i}^*}}=\underset{N\nearrow\infty}{\lim}\int_{\T}d\bz\,\overline{\mathring{\psi}(\bz)}\int_{\T}d\bz^\prime\mathfrak{K}_\T[F^{(N)}_{\hat{\xi}}]_{\mathfrak{i}^*}(\bz\,,\,\bz^\prime)\,\mathring{\phi}(\bz^\prime).
\eeq
This together with formula \eqref{DF-OpABFZ} ends the proof of point [A].

Let us now  prove point [B]. From \eqref{bpm1} and the decomposition \eqref{F-II-31} we have:
\begin{align}
	\mathfrak{K}[F^{(N)}](\bz,\bz^\prime)=\underset{\gamma\in\Gamma}{\sum}\,\underset{\kappa\in\Sigma_1}{\sum}\int_{\X^*}\hspace*{-0.3cm}d\zeta\,e^{i<\zeta,\s(\bz)-\s(\bz^\prime)-2\gamma-\kappa>}\,F^{(N)}\big([\s(\bz)+\s(\bz^\prime)+\kappa]/2,\zeta\big).
\end{align}
Recalling the Poisson summation formula 
\beq\label{F-Poisson}
\underset{k\in\Z}{\sum}e^{iakt}\,=\,\frac{2\pi}{a}\underset{p\in\Z}{\sum}\delta_0\big((2\pi/a)p-t\big)\,,\qquad\forall{a}>0\,,
\eeq
and taking $a=4\pi$ in the $d$-dimensional version of the above formula we get:
\beq\begin{split}\label{F-53}
&\mathfrak{K}[F^{(N)}_\xi](\bz,\bz^\prime)=\\
&\quad=\underset{\gamma^*\in\Gamma^*}{\sum}e^{i<(1/2)\gamma^*,\s(\bz)-\s(\bz^\prime)>}(2)^{-d}\underset{\kappa\in\Sigma_1}{\sum}\big[e^{(i/2)<\gamma^*,\kappa>}F^{(N)}_\xi\big((\s(\bz)+\s(\bz^\prime)-\kappa)/2,(1/2)\gamma^*\big)\big]\,,
\end{split}\eeq
and we can change the argument $(\s(\bz)+\s(\bz^\prime)-\kappa)/2$ into ${(\s(\bz)+\s(\bz^\prime)+\kappa)/2}$ due to the $\Gamma$-periodicity of $F^{(N)}$. 

We now use a decomposition as in \eqref{F-II-31} but for the dual indices $\gamma^*\in\Gamma^*$: 
$\gamma^*=2\alpha^*+\kappa^*$ with $\alpha^*\in\Gamma_*$ and $\kappa^*\in\Sigma_1^*$ and we obtain the following formula for the integral kernels:
\begin{align*}\numberthis\label{F-53-1}
&\mathfrak{K}[F^{(N)}_\xi](\bz,\bz^\prime)=\underset{\alpha^*\in\Gamma^*}{\sum}e^{i<\alpha^*,\s(\bz)-\s(\bz^\prime)>}\,\times\\
&\quad\times\,(2)^{-d}\underset{\kappa^*\in\Sigma_1^*}{\sum}e^{(i/2)<\kappa^*,\s(\bz)-\s(\bz^\prime)>}\underset{\kappa\in\Sigma_1}{\sum}\big[e^{(i/2)<\kappa^*,\kappa>}F^{(N)}_\xi\big((\s(\bz)+\s(\bz^\prime)-\kappa)/2,\alpha^*+(1/2)\kappa^*\big)\big]\,.
\end{align*}
This implies
\beq
\mathfrak{K}[F^{(N)}_\xi]=\underset{\alpha^*\in\Gamma^*}{\sum}\,(\hat{\chi}_{-\alpha^*}\otimes\hat{\chi}_{\alpha^*})[\Phi_{\alpha^*}[F^{(N)}_\xi]\circ(\s\otimes\s)]\,,\qquad\forall{N}\in\Nb\,,
\eeq
with
\begin{align}\label{F-53-2}
	&\Phi_{\alpha^*}[F^{(N)}_\xi](\hat{x},\hat{y})=\\ \nonumber
    &\quad=\underset{\kappa^*\in\Sigma_1^*}{\sum}e^{(i/2)<\kappa^*,\hat{x}-\hat{y}>}\underset{\kappa\in\Sigma_1}{\sum}\big[\frac{e^{(i/2)<\kappa^*,\kappa>}}{2^d}F^{(N)}_\xi\big((\hat{x}+\hat{y}+\kappa)/2,\alpha^*+(1/2)\kappa^*\big)\big].
\end{align}
One can extend the above definition to points $(\hat{x}+\alpha,\hat{y}+\beta)\in\X\times\X$ with $(\alpha,\beta)\in\Gamma\times\Gamma$, and notice the $\Gamma\times\Gamma$-periodicity by writing $\alpha+\beta=2\gamma+\sigma$ with $\gamma\in\Gamma$ and $\sigma\in\Sigma_1$ and noticing that $\alpha-\beta=2(\gamma-\beta)+\sigma$ and making the change of summation variable $\kappa\mapsto[\kappa+\sigma]\, \text{\tt mod}\,2\Gamma$ one gets:
\begin{align}
&\Phi_{\alpha^*}[F^{(N)}_\xi](\hat{x}+\alpha,\hat{y}+\beta)=\underset{\kappa^*\in\Sigma_1^*}{\sum}e^{(i/2)<\kappa^*,\hat{x}-\hat{y}+\sigma>}\,\times\\ \nonumber
&\quad\times\,\underset{\kappa\in\Sigma_1}{\sum}\big[\frac{e^{(i/2)<\kappa^*,\kappa-\sigma>}}{2^d}F^{(N)}_\xi\big((\hat{x}+\hat{y}+\kappa)/2,\alpha^*+(1/2)\kappa^*\big)\big]=\Phi_{\alpha^*}[F^{(N)}_\xi](\hat{x},\hat{y}).
\end{align}
This ends the proof of point [B].\\

In order to prove point [C], we apply the Fourier transform $\mathcal{F}_{\T}:L^2(\T)\overset{\sim}{\longrightarrow}\ell^2(\Gamma_*)$ and compute:
\begin{align*}
	&\big(\mathcal{F}_{\T}\big(\widetilde{\Op}(F^{(N)})_{\xi}\big)\mathcal{F}_{\T}^{-1}\big)_{\alpha^*,\beta^*}=\int_{\T}d\bz\int_{\T}d\bz^\prime\,e^{i<\alpha^*,\s(\bz)>}\,e^{-i<\beta^*,\s(\bz^\prime)>}\underset{\gamma^*\in\Gamma^*}{\sum}e^{i<(1/2\gamma^*,\s(\bz)-\s(\bz^\prime)>}\,\times\\\numberthis\label{F-51}
&\hspace*{3cm}\times\,(2)^{-d}\underset{\kappa\in\Sigma_1}{\sum}\big[e^{(i/2)<\gamma^*,\kappa>}F^{(N)}\big((\s(\bz)+\s(\bz^\prime)+\kappa)/2,(1/2)\gamma^*+\xi\big)\big]\,.
\end{align*}
Taking into account the $\Gamma$-periodicity of the symbol $F^{(N)}\in{S}^p_1(\Xi)_{\Gamma}$ we may consider its Fourier decomposition:
\beq\begin{split}
F^{(N)}_{\xi}(u,\zeta)=&\underset{\mu^*\in\Gamma_*}{\sum}\,\widehat{F}^{(N)}_{\mu^*}(\zeta+\xi)\,\overline{\hat{\chi}_{\mu^*}\big((\s\circ\p)(u)\big)},\\
&\widehat{F}^{(N)}_{\mu^*}(\zeta):=\int_{\T}d\bz\,e^{i<\mu^*,\s(\bz)>}\,F^{(N)}(\s(\bz),\zeta).
\end{split}\eeq
Inserting this decomposition into \eqref{F-51} we conclude that:
\begin{align*}\numberthis\label{F-54}
	&\big(\mathcal{F}_{\T}\widehat{\Op}(F^{(N)}_\xi)_{\mathfrak{i}^*}\mathcal{F}_{\T}^{-1}\big)_{\alpha^*,\beta^*}=\int_{\T}d\bz\int_{\T}d\bz^\prime\,e^{i<\alpha^*,\s(\bz)>}\,e^{-i<\beta^*,\s(\bz^\prime)>}\underset{\gamma^*\in\Gamma^*}{\sum}e^{i<(1/2)\gamma^*,\s(\bz)-\s(\bz^\prime)>}\,\times\\
	&\hspace*{3cm}\times\,(2)^{-d}\underset{\kappa\in\Sigma_1}{\sum}\big[e^{(i/2)<\gamma^*,\kappa>}\big[\underset{\mu^*\in\Gamma^*}{\sum}e^{-i<\mu^*,(\s(\bz)+\s(\bz^\prime)+\kappa)/2>}\widehat{F}^{(N)}_{\mu^*}((1/2)\gamma^*+\xi)\big]\Big]\,.
\end{align*}
The oscillating exponential factors appearing in \eqref{F-54} allow to bring a number of simplifications to it. First we notice that we may gather the factors depending on $\kappa\in\Sigma_1$ and compute:
\beq
\underset{\kappa\in\Sigma_1}{\sum}e^{(i/2)<\gamma^*-\mu^*,\kappa>}=\left\{\begin{array}{ll}
0\,, & \gamma^*-\mu^*\in\Gamma_*\setminus[2\Gamma_*]\,,\\
2^d\,,& \gamma^*-\mu^*\in[2\Gamma_*]\,.
\end{array}\right.
\eeq
Finally, we may use the inverse Fourier formula for both integrals over $\bz$ and $\bz^\prime$. The integral over $\bz$ implies that: $\alpha^*+(1/2)(\gamma^*-\mu^*)=0$ while the integral over $\bz^\prime$ implies that $\beta^*+(1/2)(\gamma^*+\mu^*)=0$, i.e.:
$
\gamma^*=-(\alpha^*+\beta^*)
$
and $\mu^*=\alpha^*-\beta^*$, concluding that:
\begin{align*}\numberthis\label{F-54-1}
	&\big(\mathcal{F}_{\T}\widehat{\Op}(F^{(N)}_\xi)_{\mathfrak{i}^*}\mathcal{F}_{\T}^{-1}\big)_{\alpha^*,\beta^*}=\widehat{F}^{(N)}_{\alpha^*-\beta^*}\big(-(1/2)(\alpha^*+\beta^*)+\xi\big).
\end{align*}
The last step consists in letting $N\in\Nb$ go to infinity and notice that for each fixed $\mu^*$ we have the estimation:
\begin{align}
&\int_{\T}d\bz\,e^{i<\mu^*,\s(\bz)>}\,\big[F^{(N)}\big(\s(\bz),\zeta\big)-F\big(\s(\bz),\zeta\big)\big]=\\ \nonumber
&\hspace*{5cm}=-\int_{\T}d\bz\,e^{i<\mu^*,\s(\bz)>}\,\theta_N(\zeta)F\big(\s(\bz),\zeta\big)\,\underset{N\nearrow\infty}{\longrightarrow}\,0,\quad\forall\zeta\in\X^*
\end{align} 
and thus, the Fourier coefficient $\widehat{F}^{(N)}_{\mu^*}(\eta)\in\mathbb{C}$ converges at the value $\widehat{F}_{\mu^*}(\eta)\in\mathbb{C}$ for any $\eta\in\X^*$.
\end{proof}

\subsubsection{Decomposition of the self-adjointness domain}\label{SSS-sa-dom}
Theorem \ref{L-Abounded} implies that for a magnetic field satisfying \eqref{H-Bper} we have the equality $\mathscr{H}^p_A(\X)=\mathscr{H}^p(\X)$. 
Then let us analyze what becomes the space $\mathscr{H}^p(\X)$ for some $p>0$ after a  Bloch type transformation. Let us recall that $\mathscr{H}^p(\X)$ is the subspace of functions $f\in L^2(\X)$ satisfying the condition $\Op(\mathfrak{m}_p)f\in L^2(\X)$ for $\mathfrak{m}_p(x,\xi):=<\xi>^p$. The decomposition \eqref{DF-OpABF},  together with the theory of direct integrals of Hilbert spaces in \cite{Dix} (see Sections II.1.3, II.1.7, II.2.3),  imply that for any $p>0$:
\beq
\U_{BF}\mathscr{H}^p(\X)\,=\,\int_{\T_*}^{\oplus}\hspace*{-2pt} d\bz^*\,\mathscr{F}^{p}_{\bz^*}\,=:\mathscr{F}^{p},
\eeq
where:
\beq
\mathscr{F}^{p}_{\bz^*}:=\big\{f\in\mathscr{F}_{\bz^*},\ \Op(\mathfrak{m}_p)f\in{L}^2_{\text{\tt loc}}(\X)\big\}.
\eeq
 Finally recalling \eqref{eq:2.17}, we notice that for any $\xi\in\X^*$ we have that $$ \mathfrak{f}(\xi,\cdot)\mathscr{F}^p_{\p_*(\xi)}\cong\mathscr{F}^p_{\mathfrak{i}^*}\cong\mathscr{H}^p(\T)$$ the Sobolev space of order $p$ on the $d$-dimensional torus. 

Using Remark \ref{R-OAF} we deduce that $\U_{BF}\overline{\Op(F)}\U_{BF}^{-1}$ and $\U_{BFZ}\overline{\Op(F)}\U_{BFZ}^{-1}$ are decomposable  (see Section II.2.3 in \cite{Dix}), lower-bounded, self-adjoint operators and thus each of their fibre operators $\widehat{\Op}(F)_{\bz^*}$ is lower-bounded and self-adjoint in $\mathscr{F}_{\bz^*}$ with domain $ \mathscr{F}^p_{\bz^*}\cong\mathfrak{f}(\s_*(\bz^*),\cdot)\mathscr{H}^p(\T)$ and resp. $\widetilde{\Op}(F)_{\xi}$ are lower bounded and self-adjoint in $L^2(\T)$ with common domain $\mathscr{H}^p(\T)$. 

\subsubsection{Regularity and spectral properties.}\label{SSS-reg}
The above arguments imply that for any $\zz\notin\sigma\big(\mO(F)\big)$ the associated resolvent operators are also decomposable. Thus we have the equality:
\beq
\big(\mO(F)-\zz\bb1\big)^{-1}=\U_{BFZ}^{-1}\big(\widetilde{\mO}(F)_{\bullet}-\zz\bb1\big)^{-1}\U_{BFZ}
\eeq
where $\big(\widetilde{\mO}(F)_{\bullet}-\zz\bb1\big)^{-1}$ is the operator of fibrewise multiplication in $\mathscr{G}$ with the operator-valued function 
\beq\label{DF-Rez-xi}
\X^*\ni\xi\mapsto\big(\widetilde{\mO}(F)_{\xi}-\zz\bb1\big)^{-1}=\Big(\widehat{\Op}\big((\bb1_{\X}\otimes\tau_{\xi})F\big)_{\mathfrak{i}^*}-\zz\bb1\Big)^{-1}\in\mathcal{L}\big(L^2(\T);\mathscr{H}^p(\T)\big).
\eeq

\begin{proposition}\label{P-PNf-2}
	The map introduced  in \eqref{DF-Rez-xi} is infinitely differentiable for the operator norm topology on $\mathcal{L}\big(L^2(\T);\mathscr{H}^p(\T)\big)$.
\end{proposition}
\begin{proof}
Due to Remark \ref{R-1} we only have to verify the smoothness around $\xi=0$; thus we have to control the behaviour of the following quotient for $\delta\rightarrow0$:
\beq
X_\delta\,:=\,\delta^{-1}\Big[\Big(\widehat{\Op}\big((\bb1_{\X}\otimes\tau_{\delta})F\big)_{\mathfrak{i}^*}-\zz\bb1\Big)^{-1}-\Big(\widehat{\Op}\big(F\big)_{\mathfrak{i}^*}-\zz\bb1\Big)^{-1}\Big]\,.
\eeq
Let us compute:
\begin{align*}
	X_\delta&=\Big(\widehat{\Op}\big((\bb1_{\X}\otimes\tau_{\delta})F\big)_{\mathfrak{i}^*}-\zz\bb1\Big)^{-1}\Big[\delta^{-1}\widehat{\Op}\big((\bb1_{\X}\otimes\tau_{\delta})F-F\big)_{\mathfrak{i}^*}\Big]\Big(\widehat{\Op}\big(F\big)_{\mathfrak{i}^*}-\zz\bb1\Big)^{-1}\\
	&\ \underset{\delta\rightarrow0}{\longrightarrow}\,\Big(\widehat{\Op}\big((\bb1_{\X}\otimes\tau_{\delta})F\big)_{\mathfrak{i}^*}-\zz\bb1\Big)^{-1}\Big[\widehat{\Op}\big(\partial_\xi\big|_{\xi=0}{F}\big)_{\mathfrak{i}^*}\Big]\Big(\widehat{\Op}\big(F\big)_{\mathfrak{i}^*}-\zz\bb1\Big)^{-1}.
\end{align*}
Let us denote by $ F_{\delta}:=(\bb1_{\X}\otimes\tau_{\delta})F$ and by $\mathcal{r}_{\zz}(F_\delta)\in{S}^{-p}_1(\Xi)_{\Gamma}$ the symbol of $\Big(\widehat{\Op}\big(F_{\delta}\big)-\zz\bb1\Big)^{-1}$ and notice that: 
\beq
\mathcal{r}_{\zz}(F_\delta)\sharp\partial_\xi\big|_{\xi=0}{F}\sharp\mathcal{r}_{\zz}(F_0)\in{S}^{-p-1}_1(\Xi)_\Gamma\,.
\eeq
Using the magnetic pseudodifferential calculus and the BFZ transformation we get the identities:
\begin{align*}
	&\Big(\widehat{\Op}\big(F_{\delta}\big)_{\mathfrak{i}^*}-\zz\bb1\Big)^{-1}\Big[\widehat{\Op}\big(\partial_\xi\big|_{\xi=0}{F}\big)_{\mathfrak{i}^*}\Big]\Big(\widehat{\Op}\big(F\big)_{\mathfrak{i}^*}-\zz\bb1\Big)^{-1}\\
	&\ =\widehat{\Op}\big(\mathcal{r}_{\zz}(F_\delta)\big)_{\mathfrak{i}^*}\widehat{\Op}\big(\partial_\xi\big|_{\xi=0}{F}\big)_{\mathfrak{i}^*}\widehat{\Op}\big(\mathcal{r}_{\zz}(F_0)\big)_{\mathfrak{i}^*}\\
	&\ =\widehat{\Op}\big(\mathcal{r}_{\zz}(F_\delta)\sharp\big(\partial_\xi\big|_{\xi=0}{F}\big)\sharp\mathcal{r}_{\zz}(F_0)\big)_{\mathfrak{i}^*}\ \in\ \mathcal{L}\big(L^2(\T);\mathscr{H}^p(\T)\big)\subset\mathbb{B}\big(L^2(\T)\big).
\end{align*}
Moreover, we also notice that:
\begin{align*}
	\Big(\big[\delta^{-1}\big(F_\delta-F\big)-(\partial_\xi\big|_{\xi=0}F)\big]\sharp\mathcal{r}_{\zz}(F_0)\,\in\,S^{-p-1}_1(\Xi)_\Gamma\subset{BC}^\infty(\Xi)
\end{align*}
and thus, for any continuous norm $\lnu:S^0_1(\Xi)\rightarrow\R_+$ we have the limits:
\begin{align*}
	\underset{\delta\rightarrow0}{\lim}\ \lnu\,\Big(\mathfrak{m}_p\sharp\mathcal{r}_{\zz}(F_\delta)\sharp\big[\delta^{-1}\big(F_\delta-F\big)-(\partial_\xi\big|_{\xi=0}F)\big]\sharp\mathcal{r}_{\zz}(F_0)\Big)\,=\,0.
\end{align*}
Iterating this argument and taking into account the smoothness of $F$ one obtains the conclusion.
\end{proof}

The above arguments, the Sobolev embedding Theorem, the properties of the Sobolev spaces on compact manifolds and the usual 'bootstrap' procedure used for eigenvectors of elliptic operators, taking into account that $\widehat{\Op}(F)_{\mathfrak{i}^*}\in\mathcal{L}\big(L^2(\T);\mathscr{H}^p(\T)\big)$ imply the following proposition.
\begin{proposition}\label{P-PNf-3}
For any $\xi\in\X^*$, the operator $\widetilde{\Op}(F)_{\xi}$  is a lower-bounded self-adjoint operator with compact resolvent and thus has a real lower-bounded discrete spectrum diverging to $+\infty$ with eigenfunctions of class $C^\infty(\T)$.
\end{proposition}

 \begin{remark}\label{R-PNf-1}
When $F\in{S}^p_1(\Xi)_{\Gamma}$ with $p>0$ is elliptic, the conclusion of Proposition \ref{P-Fzst-inv} may be completed with the consequences of the paragraphs \ref{SSS-sa-dom} and \ref{SSS-reg} and we conclude that the canonical extension of $\Op(F)$ to $\mathscr{S}^\prime(\X)$ takes each $\mathscr{F}^p_{\bz^*}$ into $\mathscr{F}_{\bz^*}$, with $\widehat{\Op}(F)_{\bz^*}\pm{i\bb1}$ being surjective maps $\mathscr{F}^p_{\bz^*}\rightarrow\mathscr{F}_{\bz^*}$.
\end{remark}

\section{A class of pseudo-differential operators on the torus}\label{S-3}

When we have obtained Formula \eqref{DF-OpABF} we have explained that in the BF representation the fibre operator at $\bz^*\in\T_*$ appearing  in the direct integral decomposition of $\U_{BF}\overline{\Op(F)}\U_{BF}^{-1}$ is the restriction to $\mathscr{F}_{\bz^*}$ of the extension to $\mathscr{S}^\prime(\X)$ of the magnetic Weyl operator $\Op(F)$ 
(defined initially on $\mathscr{S}(\X)\subset{L}^2(\X)$). We would like to consider the fibre operator $\widetilde{\Op}(F)_\xi$ as a pseudo-differential operator on the $d$-dimensional torus $\T$.

\subsection{Brief reminder of the  pseudo-differential calculus on the torus}

Considering the smooth real manifold structure on $\T\cong\Sd$ one may use the general theory of H\"{o}rmander for pseudo-differential operators on real smooth manifolds (see \cite{Ho-3}), in order to obtain a 'local calculus' that has a nice covariant behaviour under  changes of coordinates and allows for the identification of a \textit{principal symbol}. 

On the other hand, the abelian Lie group structure of $\T$ allows for a 'global theory' of pseudo-differential operators based on the general Fourier transform theory on abelian locally compact groups (see \cite{Fo-AHA, McL-91, RT-06, RT-10} and the references therein). In fact, Theorem 4.4 in \cite{McL-91} proves the equivalence of the two procedures and gives a precise correspondence between the two types of symbols: 'local symbols' defined on charts for the H\"{o}rmander theory and 'global symbols' defined on the product $\T\times\widehat{\T}$ with $\widehat{\T}=\Zd$ the Pontryagin dual of $\T$.
In our paper we shall use this 'global theory' for pseudo-differential operators on the $d$-dimensional torus, choosing the approach based on Weyl systems, that will be  presented in the following subsections, and proposing a 'symmetric' representation for the Weyl system that has a nice behaviour relative to adjunction. In order to clarify the connection with the usual pseudo-differential calculus developed in the literature, using the Kohn-Nirenberg procedure, let us very briefly recall some main facts. 

Due to the discretness of the dual group $\T_*\cong\Zd$ one has to replace the H\"{o}rmander type symbols that are smooth functions on the phase space ($\Rd\times\widehat{\Rd}$) by some similar classes of vector-valued sequences indexed by $\Zd$ and taking values in $C^\infty(\T)$. 
We need to recall the definition of some operators used in studying sequences and replacing the usual partial derivations, namely the \textit{step 1 differences} (see also \cite{RT-06, McL-91} where a different notation is used).  
Given any $\Zd$-indexed sequence $\{\vec{f}_{\gamma}\}_{\gamma\in\Zd}$ of elements  in a Fr\'{e}chet space we define:
\beq\label{F-B-1}
\forall{j}\in\underline{d}:\quad(\dd_j\vec{f})_\gamma:=\vec{f}_{\gamma+\varepsilon_j}-\vec{f}_{\gamma}.
\eeq 
On $C^\infty(\T)$ we consider the Fr\'{e}chet topology  defined by the $\sup$ norms of the derivatives that we denote by:
\beq
\lnu_n(f):=\underset{\bz\in\T}{\sup}\,\underset{|\mu|\leq{n}}{\max}\big|\big(\mathring{\partial}^\mu{f}\big)(\bz)\big|
\eeq
with $\mathring{\partial}$ the invariant vector fields defined by the Lie structure (as recalled in \eqref{DF-T-deriv}).
 Then, for $m\in\R$ and $(\rho,\delta)\in[0,1]\times[0,1]$ one defines \textit{the symbols of order $m$ and type $(\rho,\delta)$ on the $d$-dimensional torus}, as sequences of smooth functions $\Gamma_*\ni\alpha^*\mapsto{a}(\cdot,\alpha)\in{C}^\infty(\T)$ such that for any $(\mu,\nu)\in\mathbb{N}^d\times\mathbb{N}^d$, there exists ${c}(\mu,\nu)>0$ satisfying:
 \beq\label{F-3-3}
\big|\big(\mathring{\partial}^\nu\dd^\mu{a}\big)(\bz,\alpha^*)\big|\,\leq\,c(\mu,\nu)\big(1+|\alpha^*|\big)^{m-\rho|\mu|+\delta|\nu|},\quad\forall(\bz,\alpha^*)\in\T\times\Gamma_*.
 \eeq
 Then one denotes by $S^m_{\rho,\delta}(\T\times\Gamma_*)$ the complex space of these symbols. Finally, one defines the pseudo-differential operator associated with  a symbol $a\in{S}^m_{\rho,\delta}(\T\times\Gamma_*)$ as the following operator acting on ${C}^\infty(\T)$:
 \beq\label{F-3-4}
 \big(a(X,D)\varphi\big)(\bz)\,:=\,\underset{\alpha^*\in\Gamma_*}{\sum}\int_{\T}d\bz^\prime\,e^{i<\alpha^*,\s(\bz)-\s(\bz^\prime)>}\,a(\bz,\alpha^*)\,\varphi(\bz^\prime).
 \eeq

One may also consider the family of \textit{toroidal amplitudes of order $m\in\R$ and type $(\rho,\delta)\in[0,1]\times[0,1]$}, that are sequences of smooth functions of the form $\Gamma_*\ni\alpha\mapsto{a}(\cdot,\cdot,\alpha^*)\in{C}^\infty(\T\times\T)$ verifying that
 for all $(\mu,\nu,\nu')\in\mathbb{N}^d\times\mathbb{N}^d\,,$ there exists $c(\mu,\nu,\nu')>0$ such that
\beq \label{DF-PNf-1}
 \, \big|\big(\mathring{\partial}^\nu\mathring{\partial}^{\nu'}\dd^\mu{a}\big)(\bz,\bz',\alpha^*)\big|\,\leq\,c(\mu,\nu,\nu')\big(1+|\alpha^*|\big)^{m-\rho|\mu|+\delta|\nu+\nu'|},\,
\forall(\bz,\alpha^*)\in\T\times\Gamma_*\,.
\eeq
For a toroidal amplitude $\widetilde{a}$, we associate the in $\mathcal{L}\big(\mathscr{S}(\X)\big)$ defined by:
\beq\label{F-3-6}
\big(\widetilde{\Op}(a)\varphi\big)(\bz)\,:=\,\underset{\alpha^*\in\Gamma_*}{\sum}\int_{\T}d\bz^\prime\,e^{i<\alpha^*,\s(\bz)-\s(\bz^\prime)>}\,a(\bz,\bz',\alpha^*)\,\varphi(\bz^\prime),\quad\forall\varphi\in\mathscr{S}(\X).
\eeq
Theorem 5.2 in \cite{RT-10} proves that given any toroidal amplitude $a$ of order $m\in\R$ and type $(\rho,\delta)\in[0,1]\times[0,1]$, there exists a unique symbol $\sigma[a]\in{S}^m_{\rho,\delta}(\T\times\Gamma_*)$ such that $\widetilde{Op}(a)=\sigma[a](X,D)$.

\subsection{The Weyl system associated with a locally compact abelian group}

Let us very briefly recall that given an abelian locally compact group $\mathcal{G}$ (with composition denoted by $\circ$) with a Haar measure $dx$, we can associate with it its Pontryagin dual $\widehat{\mathcal{G}}$, i.e. the set of its irreducible unitary representations, that has a canonical structure of abelian locally compact group (\cite{Fo-AHA}) with Haar mesure $d\xi$ chosen with some given normalization.  
\begin{definition} We call \textit{Weyl system associated with the dual pair} $\big(\mathcal{G},\hat{\mathcal{G}}\big)$, the pair of unitary representations:
\begin{align}\label{eq:3.1}
	\big(U_{\mathcal{G}},V_{\mathcal{G}}\big):\;&\mathcal{G}\times\widehat{\mathcal{G}}\rightarrow\mathbb{U}\big(L^2(\mathcal{G})\big)\times\mathbb{U}\big(L^2(\mathcal{G})\big),\\ \nonumber
	&\big(U_{\mathcal{G}}(z)f\big)(x):=f(z\circ{x}),\ \big(V_{\mathcal{G}}(\zeta)f\big)(x):=\zeta(x)f(x),\  \forall f\in L^2(\mathcal{G},dx),
\end{align}
where each $\zeta\in\widehat{\G}$ is a smooth map $\G\rightarrow\mathbb{S}$ and we denote by $\zeta(x)$ its value in $\mathbb{S}\subset\Co$.
\end{definition}
One easily verifies from \eqref{eq:3.1}  that:
 \beq\begin{split}
\big(U_{\mathcal{G}}(z)V_{\mathcal{G}}(\zeta)U_{\mathcal{G}}(z^{-1})V_{\mathcal{G}}(\zeta^{-1})f\big)(x)&=\zeta(z\circ{x})\big(U_{\mathcal{G}}(z^{-1})V_{\mathcal{G}}(\zeta^{-1})f\big)(z\circ{x})\\
&=\zeta(z\circ{x})\zeta(x)^{-1}f(x)=\zeta(z)f(x)\,,
 \end{split}\eeq 
and thus the above unitary operators satisfy the commutation relation:
\beq
U_{\mathcal{G}}(z)V_{\mathcal{G}}(\zeta)U_{\mathcal{G}}(z^{-1})V_{\mathcal{G}}(\zeta^{-1})=\zeta(z)\bb1_{L^2(\mathcal{G})}\,.
\eeq

 \begin{definition} Given a Weyl system $(U_{\G},V_{\G}):\G\times\widehat{\G}\rightarrow\mathbb{U}\big(L^2(\G)\big)$ we call \emph{projective unitary Weyl representation} associated with it, a map $W_{\G}:\G\times\widehat{\G}\rightarrow\mathbb{U}\big(L^2(\G)\big)$
verifying the identities:
\beq\label{F-W-prepr}\begin{split}
&W_{\G}(x,1)=U_{\G}(x),\ W_{\G}(1,\xi)=V_{\G}(\xi),\\ &W_{\G}(x,\xi)W_{\G}(y,\eta)=\xi(y)^{-1}\,\eta(x)W_{\G}(y,\eta)W_{\G}(x,\xi)\,.
\end{split}\eeq
\end{definition}

We have two Weyl systems: the one associated with  the dual pair $(\Rd,\Rd)$ and the one associated with  the dual pair $(\Rd/\Zd,\Zd)$. The presence of these two Weyl systems in the study of the Bloch-Floquet theory has been observed  in \cite{dNL-11} and we intend to give a precise description of their relation in this context.

A key result in harmonic analysis defines for any pair of dual groups $\big(\mathcal{G},\widehat{\mathcal{G}}\big)$, an abstract Fourier-Pontryagin transformation (see Sections 4.1 and 4.2 in \cite{Fo-AHA}):
\beq
\big(\mathcal{F}_{\mathcal{G}}F\big)(\eta):=\int_{\mathcal{G}}dx\,\eta(x)\,F(x),\quad\forall F\in L^1(\mathcal{G}),\ \forall\eta\in\widehat{\mathcal{G}}\,,
\eeq
that extends to a unitary operator from $L^2(\mathcal G)$ onto $L^2(\widehat{\mathcal G})$.  We denote its inverse by $\mathcal{F}_{\widehat{\mathcal{G}}}$.

For each unitary Weyl representation of a Weyl pair $\big(\mathcal{G},\hat{\mathcal{G}}\big)$ one can construct a functional calculus for any continuous, compactly supported function $\Phi\in C_c(\mathcal{G}\times\widehat{\mathcal{G}})$, defining:
\beq
\Op_{\mathcal{G}}(\Phi):=\int_{\mathcal{G}\times\widehat{\mathcal{G}}}dx\,d\xi\,\big([\mathcal{F}_{\mathcal{G}}\otimes\mathcal{F}_{\hat{\mathcal{G}}}]\Phi\big)(x,\xi)\,W_{\mathcal{G}}(x,\xi)\in\mathbb{B}\big(L^2(\mathcal{G})\big).
\eeq

In defining a unitary Weyl representation for a Weyl system $\big(\mathcal{G},\widehat{\mathcal{G}}\big)$, an important aspect is to choose an order for the two unitaries $U_\G(z)$ and $V_\G(\zeta)$, or equivalently a phase function in front of the product $U_\G(z)\,V_\G(\zeta)$. Having in mind the Weyl calculus on $\Rd$ that we discussed briefly in the Introduction (see \eqref{F-10} and \eqref{F-W-R-prepr}) we would like to have a kind of "symmetric" unitary Weyl representation satisfying the properties:
\begin{align}\label{F-B-2}
	&W_{\mathcal{G}}(x,\xi)W_{\mathcal{G}}(y,\eta)=\sqrt{\xi(y)^{-1}\,\eta(x)}\, W_{\mathcal{G}}(x\circ{y},\xi\circ\eta) \\
	&\Op_{\mathcal{G}}(\Phi)^*=\Op_{\mathcal{G}}(\overline{\Phi}).
\end{align}
For that we would like to define it as$
	"\sqrt{\xi(x^{-1})}U_{\mathcal{G}}(x)V_{\mathcal{G}}(\xi)"$,
but we have to make a smooth choice for the square root and we have to deal now with this question in the case when $ (\mathcal G, \widehat {\mathcal G})=(\Sd,\Zd)$. 

\subsection{The Weyl system of the torus and its symmetric representation}

\subsubsection{The Weyl system.}
Let us apply the above abstract construction to the dual pair $ (\mathcal G, \widehat {\mathcal G})=(\Sd,\Zd)$.  As in \eqref{DF-char-zst}, we work with the isomorphism $\Zd\ni\gamma\mapsto\hat{\chi}_{\gamma}\in\widehat{\Sd}$ given explicitely by
\beq
\hat{\chi}_{\gamma}(\bz):=\bz^{-\gamma}=\underset{1\leq j\leq d}{\bigprod}\,\z_j^{-\gamma_j}\in\mathbb{S}
\eeq
and we have the unitary representations:
\begin{align}\label{DF-TW-syst}
	&\big(U_{\T},V_{\T}\big):\T\times\Gamma_*\rightarrow\mathbb{U}\big(L^2(\T)\big)\times\mathbb{U}\big(L^2(\T)\big),\\ \nonumber
	&\quad\big(U_{\T}(\bz^\circ)\mf\big)(\bz):=\mf(\bz^\circ\bz),\ \big(V_{\T}(\gamma^*)\mf\big)(\bz):=\bz^{-\gamma^*}{\mf}(\bz)\equiv\Big[\underset{1\leq j\leq d}{\prod}\z_j^{-\gamma^*_j}\Big]\mf(\bz),\  \forall\mf\in L^2(\T,d\bz).
\end{align}
We notice the following commutation relations:
\begin{align}
	\big(U_{\T}(\mbz)V_{\T}(\gamma^*)\mf\big)(\bz)&=\hat{\chi}_{\gamma^*}(\mbz\,\bz)\,\mf(\mbz\,\bz)=\hat{\chi}_{\gamma^*}(\mbz)\big(V_{\T}(\gamma^*)U_{\T}(\mbz)\mf\big)(\bz).
\end{align}

\paragraph{The generators.} Let us fix any $\mf\in{C}^\infty(\T)$ and $\mbz\in\T$ and consider the 1-parameter unitary group:
\beq
\R\ni{t}\mapsto\p\big(t\s(\mbz)\big)\in\T
\eeq
\beq
\big[(U_{\T}\circ\p)\big(t\s(\mbz)\big)\mf\big](\bz)=\mf\big[\p\big(t\s(\mbz)\big)\bz\big]=(\mf\circ\p)\big(t\s(\mbz)+\s(\bz)\big)
\eeq
having the generator:
\begin{align*}
&\partial_t\Big|_{t=0}\big[(U_{\T}\circ\p)\big(t\s(\mbz)\big)\mf\big](\bz)=\partial_t\Big|_{t=0}\big[(\mf\circ\p)\big(t\s(\mbz)+\s(\bz)\big)\big]\\
&\quad=\underset{1\leq j\leq d}{\sum}\s(\mbz)_j\big[\partial_j(\mf\circ\p)\big]\big(\s(\bz)\big)=\underset{1\leq j\leq d}{\sum}\s(\mbz)_j\big(\mathring{\partial}_j\mf\big)(\bz).
\end{align*}

For the representation $V_{\T}:\Gamma_*\rightarrow\mathbb{U}\big(L^2(\T)\big)$ we have the discrete generator:
\beq
\forall{j}\in\{1,\ldots,d\}:\quad\big[[V_{\T}(\varepsilon_j)-\bb1]\mf\big](\bz)=[\z_j^{-1}-1]\mf(\bz)
\eeq
and recalling \eqref{F-B-1} we obtain:
\beq
[V_{\T}(\varepsilon_j)-\bb1]=\mathcal{F}_{\Gamma_*}\dd_j\mathcal{F}_{\T}.
\eeq

\subsubsection{The symmetric representation.}

As discussed in connection with formula \eqref{F-B-2}, we need to work with a "smooth" version of the square root $\sqrt{\hat{\chi}_{1,\gamma^*}(\bz)}=e^{-(i/2)<\gamma^*,\s(\bz)>}$ and the arguments in the paragraph above suggest to replace $\bz\in\T\cong\Sd$ by $\tilde{\bz}\in$ and $e^{-(i/2)<\gamma^*,\s(\bz)>}$ by $e^{-i<\gamma^*/2,\s_2(\tilde{\bz})>}$ so that $\bz=\tilde{\p}_2(\tilde{\bz})=\tilde{\bz}^2$ and
\beq
\big[e^{-(i/2)<\gamma^*,\s_2(\tilde{\bz})>}\big]^2=e^{-i<\gamma^*,\s(\bz)+\kappa(\tbz)>}=e^{-i<\gamma^*,\s(\bz)>}.
\eeq

We propose to use the Riemann surface of $\z\mapsto\sqrt{\z}$ and consider a two-fold cover of $\Sb$:
\beq\label{D-Stilde2}
\widetilde{\mathbb{S}}_2:=\big\{(\z,\z^2)\in[\mathbb{C}\setminus\{0\}]^2,\,\forall\z\in\mathbb{S}\big\},\quad\mathfrak{r}_2:\widetilde{\mathbb{S}}_2\ni(\z,\z^2)\mapsto\z^2\in\mathbb{S},
\eeq
as the $\Z_2$-principal bundle $\mathfrak{r}_2:\widetilde{\mathbb{S}}_2\repi\mathbb{S}$ over $\mathbb{S}$,
having the fibre $\mathfrak{r}_2^{-1}(\mathring{\z})=\{\pm\sqrt{\mathring{\z}}\}\cong\Z_2$. We notice that although $\widetilde{\mathbb{S}}_2\approx\mathbb{S}\times\Z_2$ as sets, topologically $\widetilde{\mathbb{S}}_2$ is homeomorphic with $\mathbb{S}$. 

Let $[2\Zd]$ be the lattice in $\Rd$ based on
the vectors
$\{2\e_1,\ldots,2\e_d\}$ and consider the associated quotient group $\Sd_2:=\Rd/[2\Zd]$. Notice that we have the following explicit form of the 
canonical quotient projection $\p_2:\Rd\repi\mathbb{S}^d_2$:
$
\p_2(x)\,=\,\big(e^{i\pi x_1},\ldots,e^{i\pi x_d}\big)\in\mathbb{S}^d_2
$. The evident injection $[2\Zd]\subset\Zd$, as even integers, allows us to write the following diagram:
\beq\label{BD-2}
\dgARROWLENGTH=0.5em
\begin{diagram}
	\node{} \node{} \node{\Z_2^d} \arrow{s,r}{\hookrightarrow} \\
	\node{[2\Zd]} \arrow{s,l}{\hookrightarrow} \arrow{e,t}{\hookrightarrow}
	\node{\Rd} \arrow{s,l,=}{} \arrow{e,t}{\p_2} \node{\Sd_2}  \arrow{s,r}{\tilde{\p}_2}\\
	\node{\Zd}  \arrow{e,t}{\hookrightarrow} \arrow{s,l}{\p_2} \node{\Rd} \arrow{e,t}{\p} \node{\Sd}\\
	\node{\Z_2^d}
\end{diagram}\eeq
where $\tilde{\p}_2\circ\p_2=\p$
and thus:
$
\tilde{\p}_2(\bz)=\bz^2,
$ and
\beq
\ker\tilde{\p}_2=\p_2(\Zd)=\Zd/[2\Zd]\cong\Z_2^d\subset\Sd_2\,.
\eeq 
In conclusion, we may identify the surjective homomorphism $\tilde{\p}_2:\Sd_2\repi\Sd$ with the principal bundle $\mathfrak{r}_2^d:\widetilde{\mathbb{S}}_2^d\repi\mathbb{S}^d$ defined by the product of $d$ copies of the Riemann surface of the square-root function.
Let us also define the associated discontinuous section:
\beq\label{F-s2}
\s_2:\,\Sd_2\rightarrow\Rd,\quad\text{such that}\quad\p_2\circ\s_2=\Id_{\Sd_2}.
\eeq

Let us notice that $\s_2(\Sd_2)$ is equal to the unit cell of the lattice $[2\Zd]$ with respect to the decomposition $\Rd\ni{x}\mapsto\iota_2(x)+\tilde{x}$ with $\iota_2(x):=\iota(x/2)\in[2\Zd]$ and $\tilde{x}\in[2\E]=[-1,1)^d$.
If we denote by $\tilde{x}=(\tilde{x}_1,\ldots,\tilde{x}_d)\in\Rd$ the points in $\s_2(\Sd_2)=[2\E]\subset\Rd$ and by $\hat{x}=(\hat{x}_1,\ldots,\hat{x}_d)$ the points in $\s(\Sd)=\E\subset\Rd$ we notice that given $\tilde{x}\in\s_2(\Sd_2)\subset\Rd$ there exists a unique decomposition $\tilde{x}=\hat{x}+\kappa$ with $\hat{x}\in\s(\Sd)\subset\Rd$ and $\kappa\in\big(\Zd/[2\Zd]\big)\cong\Z_2^d$. Thus, if we write this decomposition as 
\beq\label{F-par-Sd2}
\s_2(\tilde{\bz})=\mathfrak{j}_2(\tilde{\bz})+\kappa(\tilde{\bz})\quad\text{with}\quad \mathfrak{j}_2(\tilde{\bz})\in\E,\ \kappa(\tilde{\bz})\in\Sigma_1,
\eeq
we can use this parametrization for $\Sd_2$, and notice that:
\beq
(\p\circ\mathfrak{j}_2)(\tilde{\bz})=(\p\circ\s_2)(\tilde{\bz})=(\tilde{\p}_2\circ\p_2\circ\s_2)(\tilde{\bz})=\tilde{\p}_2(\tilde{\bz})=(\tilde{\bz})^2\,\Longrightarrow\,\mathfrak{j}_2(\tilde{\bz})={\s((\tilde{\bz})^2)\,,}
\eeq 

\begin{notation}
Let us denote by $\T_2:=\X/[2\Gamma]$ that is  isomorphic with $\Sd_2$.
\end{notation}
As the two Pontryagin duals of $\T$ and $\T_2$ are isomorphic with $\Zd$, in the following arguments we shall use the following two isomorphisms:
\beq\label{DF-char}\begin{split}
	&\Zd\ni\nu\mapsto\hat{\chi}_{1,\nu}\in\widehat{\Sd},\quad\hat{\chi}_{1,\nu}(\bz):=\bz^{-\nu}=e^{-i<\nu,\s(\bz)>}\,,\\
	&\Zd\ni\nu\mapsto\hat{\chi}_{2,\nu}\in\widehat{\Sd_2},\quad\hat{\chi}_{2,\nu}(\tilde{\bz}):=\tilde{\bz}^{-\nu}=e^{-i<(1/2)\nu,\s_2(\tilde{\bz})>}\,.
\end{split}\eeq

\begin{definition}\label{D-s-tor-W-syst}
We call   \emph{symmetric Weyl representation of  $\T\times\Gamma_*$} the unitary valued smooth map (see also \eqref{DF-char} for the notations):
	\beq\label{DF-s-tor-W-syst}
	 \T_2\times\Gamma_*\ni\,(\tilde{\bz},\gamma^*)\,\mapsto\,\widetilde{\mathbb{W}}_{\T}(\tilde{\bz},\gamma^*):=\hat{\chi}_{2,-\gamma^*}(\tilde{\bz})U_{\T}\big(\tilde{\p}_2(\tilde{\bz})\big)V_{\Gamma_*}(\gamma^*)\in\mathbb{U}\big(L^2(\T)\big).
	\eeq
\end{definition}
\begin{proposition}
The symmetric Weyl system { $\widetilde{\mathbb{W}}_{\T}:\T_2\times\Gamma_*\rightarrow\mathbb{U}\big(L^2(\T)\big)$ given in Definition \ref{D-s-tor-W-syst}} has the following properties:
\beq
\widetilde{\mathbb{W}}_{\T}(\tilde{\bz},\gamma^*)^*=\widetilde{\mathbb{W}}_{\T}(\tilde{\bz}^{-1},-\gamma^*);
\eeq
\beq\label{F-comp-torWsyst}
\widetilde{\mathbb{W}}_{\T}(\tilde{\bz},\alpha^*)\widetilde{\mathbb{W}}_{\T}(\tilde{\bz'},\beta^*)=e^{(i/2)[<\alpha^*,\s_2(\tilde{\bz}')>-<\beta^*,\s_2(\tilde{\bz})>]}\widetilde{\mathbb{W}}_{\T}(\tilde{\bz}\tilde{\bz'},\alpha^*+\beta^*).
\eeq
\end{proposition}
\begin{proof}
We have the following equalities:
\begin{align*}
	\widetilde{\mathbb{W}}_{\T}(\tilde{\bz},\gamma^*)^*&=\hat{\chi}_{2,\gamma^*}(\tilde{\bz})V_{\Gamma_*}(\gamma^*)^*U_{\T}(\tilde{\bz}^{2})^*=\hat{\chi}_{2,\gamma^*}(\tilde{\bz})\hat{\chi}_{1,-\gamma^*}(\tilde{\bz}^{2})^*U_{\T}(\tilde{\bz}^{2})^*V_{\Gamma_*}(\gamma^*)^*\\
	&=\exp\big((i/2)<\gamma^*,\s_2(\tilde{\bz})>-i<\gamma^*,\s(\tilde{\bz}^2)>\big)U_{\T}(\tilde{\bz}^{-2})V_{\Gamma^*}(-\gamma^*),
\end{align*}
\begin{align*}
e^{(i/2)<\gamma^*,\s_2(\tilde{\bz})>-i<\gamma^*,\s(\tilde{\bz}^2)>}&=e^{(i/2)<\gamma^*,\s(\tilde{\bz}^2)+\kappa(\tilde{\bz})>-i<\gamma^*,\s(\tilde{\bz}^2)>}\\
&=e^{-(i/2)<\gamma^*,\s(\tilde{\bz}^2)>+(i/2)<\gamma^*,\kappa(\tilde{\bz})>}\\
&=e^{-(i/2)<\gamma^*,\s(\tilde{\bz}^2)>-(i/2)<\gamma^*,\kappa(\tilde{\bz})>}\\
&=e^{-(i/2)<\gamma^*,\s_2(\tilde{\bz})>}=\hat{\chi}_{2,\gamma^*}(\tilde{\bz}^{-1}).
\end{align*}
Thus, we have $$\widetilde{\mathbb{W}}_{\T}(\tilde{\bz},\gamma^*)^*=\widetilde{\mathbb{W}}_{\T}(\tilde{\bz}^{-1},-\gamma^*)\,.$$

We also obtain that:
\begin{align*}
	\widetilde{\mathbb{W}}_{\T}(\tilde{\bz},\alpha^*)\tilde{\mathbb{W}}_{\T}(\tilde{\bz'},\beta^*)&=\hat{\chi}_{2,-\alpha^*}(\tilde{\bz})U_{\T}(\tilde{\bz}^2)V_{\Gamma_*}(\alpha^*)\hat{\chi}_{2,-\beta^*}(\tilde{\bz'})U_{\T}(\tilde{\bz'}^2)V_{\Gamma_*}(\beta^*)\\ \nonumber
&=e^{i<\alpha^*,\s(\tilde{\bz'}^2)>}\,e^{(i/2)<\alpha^*,\s_2(\tilde{\bz})>}\,e^{(i/2)<\beta^*,\s_2(\tilde{\bz}')>}\,e^{-(i/2)<\alpha^*+\beta^*,\s_2(\tilde{\bz}\,\tilde{\bz'})>}\widetilde{\mathbb{W}}_{\T}(\tilde{\bz}\tilde{\bz'},\alpha^*+\beta^*)\\ \nonumber
	&=e^{(i/2)[<\alpha^*,\s_2(\tilde{\bz}')>-<\beta^*,\s_2(\tilde{\bz})>]}\widetilde{\mathbb{W}}_{\T}(\tilde{\bz}\tilde{\bz'},\alpha^*+\beta^*).
\end{align*}
\end{proof}

\subsection{The toroidal symmetric quantization}

In our paper, while keeping the same classes of symbols as in \cite{RT-10,RT-06, McL-91}, we consider a 'symmetric' quantization on the torus. Nevertheless, let us mention that the asymptotic expansions  procedure and the other results in \cite{RT-06} allow to prove the equivalence of the two calculi. Moreover, the result in \cite{McL-91} allows to prove the equivalence with the 'local' theory of H\"{o}rmander on the smooth manifold $\T\cong\Sd$.

We shall follow \cite{RT-06, McL-91} and consider the following complex linear spaces. 

	\paragraph{The space $\mathbf{\mathcal{o}\big(\Gamma_*;C^\infty(\T)\big)}$} is the complex linear space of all sequences $\vec{f}:=(\vec{f}_\gamma)_{\gamma\in\Zd}$ of elements $\vec{f}_\gamma\in C^\infty(\T)$, such that for any $\vec{f}\in\mathcal{o}\big(\Gamma_*;C^\infty(\T)\big)$ there exists some $p\geq0$ (depending on $\vec{f}$) verifying the estimations:
	\beq\label{DF-wp}
	\mathcal{w}^{(p)}_{m,n}(\vec{f})\,:=\,\underset{\gamma\in\Gamma_*}{\sup}<\gamma>^{-p}\,\underset{|\alpha|\leq{m}}{\max}\,\nu_n\big((\dd^\alpha\vec{f})_{\gamma}\big)\,<\,\infty,\quad\forall (m,n)\in\mathbb{N}^2.
	\eeq
It is clearly an inductive limit of Fr\'{e}chet spaces. 

	\paragraph{The space $\boldsymbol{\mathcal{o}^p_\rho\big(\Gamma_*;C^\infty(\T)\big)}$} for $p\in\R$ and $\rho=0,1$ is the complex linear space of all sequences $\vec{f}:=(\vec{f}_\gamma)_{\gamma\in\Gamma_*}$ of elements $\vec{f}_\gamma\in C^\infty(\T)$, such that for any $\vec{f}\in\mathcal{o}\big(\Gamma_*;C^\infty(\T)\big)$ and any $\alpha\in\mathbb{N}^d$:
	\beq\label{DF-wp-rho}
	\mathcal{w}^{(p,\rho)}_{m,n}(\vec{f})\,:=\,\underset{\gamma\in\Gamma_*}{\sup}<\gamma>^{-p+m\rho}\,\underset{|\alpha|\leq{m}}{\max}\,\nu_n\big((\dd^\alpha\vec{f})_{\gamma}\big)\,<\,\infty,\quad\forall (m,n)\in\mathbb{N}^2.
	\eeq
	
	\paragraph{The space $\mathbf{\mathcal{s}\big(\Gamma_*;C^\infty(\T)\big)}$} is the complex linear space of all sequences $\vec{f}:=(\vec{f}_\gamma)_{\gamma\in\Gamma_*}$ of elements $\vec{f}_\gamma\in C^\infty(\T)$, such that:
	\beq\label{DF-s}
	\forall (N,m,n)\in\mathbb{N}^3:\quad\mathcal{w}^{(-n)}_{m,n}(\vec{f})\,<\,\infty.
	\eeq
The last two families of spaces are clearly  Fr\'{e}chet spaces.

In view of the form of the symmetric Weyl representation of $\T\times\Gamma_*$ given in Definition \ref{D-s-tor-W-syst} and the fact that the Pontryagin dual of $\T_2$ is canonically isomorphic to $(1/2)\Gamma_{*}\cong\Zd$, one could consider to quantize symbols of class $\mathcal{o}\big((1/2)\Gamma_{*};C^\infty(\T)\big)$ as an analogue of the H\"{o}rmander class $S^\infty_0(\Xi )$, with $\mathcal{s}\big((1/2)\Gamma_{*};C^\infty(\T)\big)$ as family of 'regular' symbols, analogue to $S^{-\infty}(\Xi )$. Nevertheless, considering the isomorphisms described in \eqref{DF-char}, we shall use only $\Gamma_*$-indexed sequences of distributions on the $d$-dimensional torus.

\paragraph{The toroidal Fourier transform.}
Given $\mathring{F}\in \mathcal{o}\big(\Gamma_*;C^\infty(\T)\big)$, we define its 'toroidal Fourier transform'  $\mathcal{F}_{\T\times\Gamma_*}\,\mathring{F}$ as being the following $\Gamma_*$-indexed, rapidly decaying sequence of distributions on the 'double' torus $\T_2$ (with convergence in the weak distribution topology):
\beq\begin{split}
	\big(\mathcal{F}_{\T\times\Gamma_*}\,\mathring{F}\big)_{\alpha^*}(\tilde{\bz}):&=2^{-d}\hspace*{-5pt}\int_{\T}\hspace*{-3pt}d\bz'\hat{\chi}_{1,-\alpha^*}(\bz')\hspace*{-4pt}\underset{\gamma^*\in\Gamma_*}{\sum}\hat{\chi}_{2,\gamma^*}(\tilde{\bz})\mathring{F}_{\gamma^*}(\bz')\\
	&=2^{-d}\hspace*{-5pt}\int_{\T}\hspace*{-3pt}d\bz'e^{i<\alpha^*,\s(\bz')>}\hspace*{-6pt}\underset{\gamma^*\in\Gamma_*}{\sum}\hspace*{-5pt}e^{-(i/2)<\gamma^*,\s_2(\tilde{\bz})>}\,\mathring{F}_{\gamma^*}(\bz'),\quad\forall\tilde{\bz}\in\T_2.
\end{split}\eeq

Given $\mathring{F}\in \mathcal{s}\big(\Gamma^*;C^\infty(\T)\big)$,  we define its \textit{toroidal symmetric quantization} $\Op_{\T}(\mathring{F})$  as the following bounded operator in $L^2(\T)$:
\beq\begin{split}\label{DF-OpTF}
	\Op_{\T}(\mathring{F}):&=2^{-d}\underset{\alpha^*\in\Gamma_*}{\sum}\int_{\Sd_2}d\tilde{\bz}\,\Big(\big(\mathcal{F}_{\T\times\Gamma_*}\mathring{F}\big)_{\alpha^*}(\tilde{\bz})\Big)\,\widetilde{\mathbb{W}}_{\T}(\tilde{\bz},\alpha^*).
\end{split}\eeq
In this case $\big(\mathcal{F}_{\T\times\Gamma_*}\mathring{F}\big)_{\alpha^*}(\tilde{\bz})$ belongs to $\mathcal{s}\big(\Gamma_*;C^\infty(\Sd_2)\big)$ so that both, the series indexed by $\Gamma^*$ and the integral over $\Sd_2$ are absolutely convergent in the topology of $\mathbb{B}\big(L^2(\T)\big)$.

\begin{remark}\label{R-ker-Op-Hsymb}
	 Repeating the arguments in the proof of Theorem 18.1.6 in \cite{Ho-3} on sequences indexed by $\Gamma_*$ instead of H\"{o}rmander symbols allows us to see that given any ${\mathring{F}\in\mathcal{o}\big(\Gamma_*;C^\infty(\T)\big)}$ its symmetric toroidal Weyl quantization is an operator in $\mathcal{L}\big(C^\infty(\T);\mathscr{D}^\prime(\T)\big)$, thus by the Schwartz Kernel Theorem, having a distribution kernel in $\mathscr{D}^\prime(\T\times\T)$.
\end{remark}

 \subsection{The distribution kernel }

Let us compute the integral kernel of the symmetric toroidal Weyl operator $\Op_{\T}(\mathring{F})$ for some $\mathring{F}\in \mathcal{s}\big(\Gamma^*;C^\infty(\T)\big)$; the integrals and series that appear are all absolutely convergent and we shall interchange them as needed in the computations that follow. We can write  for some test function $\varphi\in C^\infty(\T)$, using the isomorphisms \eqref{DF-char} :
\begin{align}\label{F-OpT-Fcirc-vphi}
	\big(\Op_{\T}(\mathring{F})\,\varphi\big)(\bz)&=2^{-2d}\underset{\alpha^*\in\Gamma_*}{\sum}\int_{\T_2}d\tilde{\bz}'\hspace*{-2pt}\underset{\gamma^*\in\Gamma_*}{\sum}\int_{\T}d\bz''\,\mathring{F}_{\gamma^*}(\bz'')\,\times \\ \nonumber
	&\times\,\hat{\chi}_{1,-\alpha^*}(\bz'')\hat{\chi}_{2,\gamma^*}(\tilde{\bz}')\hat{\chi}_{2,-\alpha^*}(\tilde{\bz}')\hat{\chi}_{1,\alpha^*}((\tilde{\bz}')^2\bz)\,\varphi((\tilde{\bz}')^2\,\bz).
\end{align}
Let us analyze the product of characters:
\begin{align}\nonumber
\hat{\chi}_{1,-\alpha^*}(\bz'')\hat{\chi}_{2,\gamma^*}(\tilde{\bz}')\hat{\chi}_{2,-\alpha^*}(\tilde{\bz}')\hat{\chi}_{1,\alpha^*}((\tilde{\bz}')^2\bz)&=\hat{\chi}_{1,-\alpha^*}(\bz'')\hat{\chi}_{2,\gamma^*}(\tilde{\bz}')\hat{\chi}_{2,-\alpha^*}(\tilde{\bz}')\hat{\chi}_{1,\alpha^*}((\tilde{\bz}')^2)\hat{\chi}_{1,\alpha^*}(\bz)\\ \label{F-char}
&\hspace*{-1cm}=\hat{\chi}_{1,-\alpha^*}(\bz'')\hat{\chi}_{1,\alpha^*}(\bz)\hat{\chi}_{2,\gamma^*}(\tilde{\bz}')\hat{\chi}_{2,-\alpha^*}(\tilde{\bz}')\hat{\chi}_{1,\alpha^*}((\tilde{\bz}')^2)\,.
\end{align}
From \eqref{F-char} we deduce that $\hat{\chi}_{1,\alpha^*}\circ\tilde{\p}_2=\hat{\chi}_{2,2\alpha^*}$ and thus we may continue the above calculus obtaining that the product in \eqref{F-char} is equal to:
\begin{align}\label{F-char-2}
\hat{\chi}_{1,-\alpha^*}(\bz'')\hat{\chi}_{1,\alpha^*}(\bz)\hat{\chi}_{2,\gamma^*}(\tilde{\bz}')\hat{\chi}_{2,-\alpha^*}(\tilde{\bz}')\hat{\chi}_{2,2\alpha^*}(\tilde{\bz}')=\hat{\chi}_{1,\alpha^*}(\bz\,(\bz'')^{-1})\hat{\chi}_{2,\alpha^*+\gamma^*}(\tilde{\bz}').
\end{align}
We can now  go  back to \eqref{F-OpT-Fcirc-vphi} and notice that after replacing the above result \eqref{F-char-2}, we have:
\begin{align*}
	&\underset{\alpha^*\in\Gamma_*}{\sum}\int_{\T_2}d\tilde{\bz}'\hspace*{-2pt}\underset{\gamma^*\in\Gamma_*}{\sum}\int_{\T}d\bz''\,\mathring{F}_{\gamma^*}(\bz'')\,e^{i<\alpha^*,\s(\bz'')>}\,e^{-(i/2)<\gamma^*,\s_2(\tilde{\bz}')>}\,e^{(i/2)<\alpha^*,\s_2(\tilde{\bz}')>}\,e^{-i<\alpha^*,\s((\tilde{\bz}')^2\bz)>}\,\varphi((\tilde{\bz}')^2\,\bz)\\
	&\quad=\underset{\gamma^*\in\Gamma_*}{\sum}\underset{\alpha^*\in\Gamma_*}{\sum}\int_{\T_2}d\tilde{\bz}'\hat{\chi}_{2,\alpha^*+\gamma^*}(\tilde{\bz}')\varphi((\tilde{\bz}')^2\,\bz)\int_{\T}\hspace*{-2pt}d\bz''\,\hat{\chi}_{1,\alpha^*}(\bz(\bz'')^{-1})\mathring{F}_{\gamma^*}(\bz'')\\ \numberthis \label{F-60}
	&\quad=\underset{\gamma^*\in\Gamma_*}{\sum}\int_{\T}d\bz''\,\mathring{F}_{\gamma^*}(\bz'')\underset{\alpha^*\in\Gamma_*}{\sum}\,\hat{\chi}_{1,\alpha^*}(\bz(\bz'')^{-1})\int_{\T_2}d\tilde{\bz}'\hspace*{-2pt}\hat{\chi}_{2,\alpha^*}(\tilde{\bz}')\,\hat{\chi}_{2,\gamma^*}(\tilde{\bz}')\varphi((\tilde{\bz}')^2\,\bz)\,.
\end{align*}
We notice that that we may denote $(\tbz^\prime)^2=:\bz_\circ$ and use the decomposition \eqref{F-par-Sd2} in order to write:
\begin{align}
&\int_{\T_2}\hspace*{-2pt}d\tilde{\bz}^\prime\,\hat{\chi}_{2,\alpha^*}(\tilde{\bz}^\prime)\,\hat{\chi}_{2,\gamma^*}(\tilde{\bz}^\prime)\varphi((\tilde{\bz}^\prime)^2\,\bz)=\int_{\T_2}\hspace*{-2pt}d\tilde{\bz}^\prime\,e^{-i<(1/2)(\alpha^*+\gamma^*)\,,\,\s_2(\tbz^\prime)>}\,\varphi((\tilde{\bz}^\prime)^2\,\bz)\\
&\quad=\underset{\kappa\in\Sigma_1}{\sum}\,\int_{\T}\hspace*{-2pt}d\bz_\circ\,e^{-i<(1/2)(\alpha^*+\gamma^*)\,,\,\s(\bz_\circ)+\kappa>}\,\varphi(\bz_\circ\,\bz)\,.
\end{align}
We make now the change of variable $\alpha^*\mapsto\beta^*:=\alpha^*+\gamma^*$ and compute:
\begin{align*}
&\underset{\alpha^*\in\Gamma_*}{\sum}\,\hat{\chi}_{1,\alpha^*}(\bz(\bz'')^{-1})\int_{\T_2}d\tilde{\bz}^\prime\,\hat{\chi}_{2,\alpha^*}(\tilde{\bz}')\,\hat{\chi}_{2,\gamma^*}(\tilde{\bz}')\varphi((\tilde{\bz}')^2\,\bz)=\hat{\chi}_{1,-\gamma^*}(\bz(\bz'')^{-1})\,\times\\
&\qquad\times\,\underset{\beta^*\in\Gamma_*}{\sum}\,\underset{\kappa\in\Sigma_1}{\sum}\,e^{-i<(1/2)\beta^*\,,\,\kappa>}\,e^{-i<\beta^*\,,\,\s(\bz\,(\bz^{\prime\prime})^{-1})>}\int_{\T}\hspace*{-2pt}d\bz_\circ\,e^{-i<\beta^*\,,\,(1/2)\s(\bz_\circ)>}\,\varphi(\bz_\circ\,\bz)\\
&\quad=2^d\hat{\chi}_{1,-\gamma^*}(\bz(\bz'')^{-1})\,(\varphi\circ\p)\big(2s(\bz^{\prime\prime}\,\bz^{-1})+s(\bz)\big)=2^d\hat{\chi}_{1,-\gamma^*}(\bz(\bz'')^{-1})\,\varphi\big((\bz^{\prime\prime})^2\,\bz({-1}\big)
\end{align*}
and finally \eqref{F-60} becomes:
\begin{align*}
&\quad=\underset{\gamma^*\in\Gamma_*}{\sum}\int_{\T}d\bz''\,\mathring{F}_{\gamma^*}(\bz'')\underset{\alpha^*\in\Gamma_*}{\sum}\,\hat{\chi}_{1,\alpha^*}(\bz(\bz'')^{-1})\int_{\T_2}d\tilde{\bz}'\hspace*{-2pt}\hat{\chi}_{2,\alpha^*}(\tilde{\bz}')\,\hat{\chi}_{2,\gamma^*}(\tilde{\bz}')\varphi((\tilde{\bz}')^2\,\bz)\\
&\quad=2^d\underset{\gamma^*\in\Gamma_*}{\sum}\int_{\T}d\bz''\,\mathring{F}_{\gamma^*}(\bz'')\,\hat{\chi}_{1,-\gamma^*}(\bz(\bz'')^{-1})\,\varphi\big((\bz^{\prime\prime})^2\,\bz^{-1}\big)\,.
\end{align*}

 We would like to put into evidence the variable $\bz':=(\bz'')^2\cdot(\bz)^{-1}$ that is the argument of the test function $\varphi\in\mathscr{S}(\T)$ and thus we have to make the change of variables $\bz''\mapsto\bz':=(\bz'')^2\cdot(\bz)^{-1}$. We notice that this map is not bijective, in fact it is surjective but not injective. More precisely, for any $\bz'\in\T$ we have exactly $2^d$ values $\bz''\in\T$ verifying the equation $\bz^\prime=(\bz'')^2\cdot(\bz)^{-1}$ and they can be written as $\bz^{\prime\prime}_\kappa=\p(\kappa/2)\bz_\circ$ with $\kappa\in\Sigma_1$ and $\s(\bz_\circ)=(1/2)[\s(\bz^\prime)-\s(\bz)]$. Thus, denoting by $\mathring{\bz}:=\bz_\circ^2$, Formula \eqref{F-OpT-Fcirc-vphi} becomes:
\begin{align}
	&\big(\Op_{\T}(\mathring{F})\,\varphi\big)(\bz)=2^{-d}\underset{\gamma^*\in\Gamma_{*}}{\sum}\,\underset{\kappa\in\Sigma_1}{\sum}\int_{\T}d\mathring{\bz}\,\mathring{F}_{\gamma^*}(\bz_\kappa\hspace*{-6pt}'')\,e^{-i<\gamma^*,\s(\bz_\kappa\hspace*{-3pt}''\hspace*{2pt}\cdot\bz^{-1})>}\,\varphi\big(\mathring{\bz}\cdot\bz^{-1}\big)\,.
\end{align}
Finally, we can change the integration variable: $\mathring{\bz}\mapsto\bz^\prime:=\mathring{\bz}\bz^{-1}$ and write:
\begin{align}
&e^{-i<\gamma^*,\s(\bz_\kappa\hspace*{-3pt}''\hspace*{2pt}\cdot\bz^{-1})>}=e^{-i<\gamma^*,[\s(\bz_\kappa\hspace*{-3pt}''\hspace*{2pt})-\s(\bz)]>}=e^{-i<\gamma^*,[(1/2)\s(\bz^\prime\bz)+(1/2)\kappa-\s(\bz)]>},\\ \nonumber
&\s(\bz^\prime\bz)=\s(\bz^\prime)+\s(\bz)+\mathfrak{n}(\bz^\prime,\bz),\quad\mathfrak{n}(\bz^\prime,\bz)\in\Sigma_1
\end{align}
and redefining the summation index $\kappa\in\Sigma_1$ to $\kappa^\prime:=\kappa+\mathfrak{n}(\bz^\prime,\bz)$ we finally get:
\begin{align}\label{F-KOpT}
	&\big(\Op_{\T}(\mathring{F})\,\varphi\big)(\bz)=:2^{-d}\int_{\T}d\bz'\,\mathfrak{K}_{\T}[\mathring{F}](\bz,\bz')\,\varphi(\bz')=\\ \nonumber
	&=\int_{\T}d\bz'\,\underset{\gamma^*\in\Gamma_{*}}{\sum}e^{i<(1/2)\gamma^*,\s(\bz)-\s(\bz')>}\Big[\underset{\kappa\in\Sigma_1}{\sum}e^{i<(1/2)\gamma^*,\kappa>}\,\mathring{F}_{\gamma^*}\big(\p[(1/2)(\s(\bz')+\s(\bz)+\kappa)\big)\Big]\varphi\big(\bz'\big).
\end{align}

\begin{remark}
From the above formula \eqref{F-KOpT} it is evident that for any $\mathring{F}\in\mathcal{o}\big(\Gamma_{*};C^\infty(\T)\big)$ we have that $$\Op_{\T}(\mathring{F})^*=\Op_{\T}(\overline{\mathring{F}})\,.$$
\end{remark}

\begin{theorem}
	For $F\in{S}^p_1(\Xi)_{\Gamma_*}$ the BFZ representation of $\Op(F)$ has a smooth decomposition, the fibre operator $\widetilde{\Op}(F)_\xi$ above any point $\xi\in\T_*$ being a symmetric Weyl operator in the sense of \eqref{DF-OpTF} and we have the equality: $\widetilde{\Op}(F)_\xi=\Op_{\T}\big(F_\xi\circ(\s\otimes\mathfrak{j}_{*})\big)$ with $\mathfrak{j}_*:(1/2)\Gamma_*\hookrightarrow\X^*$ the inclusion map (i.e. the symbol $F_\xi$ considered on the $d$-dimensional torus $\T$ and with the second variable restricted to $(1/2)\Gamma_*\subset\X^*$).
\end{theorem}
\begin{proof}
The functions $\Phi_{\alpha^*}[F^{(N)}_\xi]$ in \eqref{F-53-2} verify the conditions \eqref{DF-PNf-1} for a toroidal amplitude and thus \eqref{F-53} becomes equivalent to \eqref{F-3-6} and defines a pseudo-differential operator on the $d$-dimensional torus $\T$. Moreover, we may consider the formula \eqref{F-53} in the limit $N\nearrow\infty$ and compare it with \eqref{F-KOpT}. We notice that for any $N\in\Nb$, the formula \eqref{F-53} may be considered as a discrete Fourier transform ${\cal{s}}\big((1/2)\Gamma^*;C^\infty(\T\times\T)\big)$ with $(1/2)\Gamma^*\cong\widehat{\T}_2$. In fact we clearly have $\s(\bz)-\s(\bz^\prime)\in\T_2$. Moreover it is rather easy to verify that $\Phi_{\cdot}[F_\xi]\in{\cal{o}}^p_1\big((1/2)\Gamma^*;C^\infty(\T\times\T)\big)$ and that we have the convergence
\[
\underset{N\nearrow\infty}{\lim}\Phi_{\cdot}[F^{(N)}_\xi]\big(\s(\bz),\s(\bz^\prime)\big)\,=\,\Phi_{\cdot}[F_\xi]\big(\s(\bz),\s(\bz^\prime)\big)\qquad\text{in}\qquad\mathcal{o}^p_1\big((1/2)\Gamma^*;C^\infty(\T\times\T)\big).
\]
\end{proof}

\begin{remark}\label{R-Final}
Let us denote by $\mathfrak{S}^p_1\repi\T_*$ the vector bundle associated to the principal bundle $\X^*\repi\T_*$ and the natural representation of $\Gamma_*$ by translations on the space of sequences ${\cal{o}}^p_1\big((1/2)\Gamma_*;C^\infty(\T)\big)$, given by $\big(\mathring{U}(\gamma^*)\mathring{F}\big)_{\alpha^*}:=\mathring{F}_{\alpha^*+\gamma^*}\in{C}^\infty(\T)$. Then any $F\in{S}^p_1(\Xi)_{\Gamma_*}$ defines a smooth section in $\mathring{F}\in\mathfrak{S}^p_1\repi\T_*$ given by $[\mathring{F}_{\bz^*}]_{(1/2)\gamma^*}(\bz):=F\big(\s(\bz),(1/2)\gamma^*+\s_*(\bz^*)\big)$. We may thus view the BFZ representations of the $\Gamma$-periodic Weyl operators as a Weyl calculus associated to the principal bundle $\X\repi\T$ with symbols of class $\mathfrak{S}^p_1\repi\T_*$ and symmetric Weyl representatiion \eqref{DF-s-tor-W-syst}.
\end{remark}

\noindent{\bf Acknowledgements.} HC acknowledges support from Grant DFF 5281-00046B of the Independent Research Fund Denmark $|$ Natural Sciences. RP acknowledges support from a grant of the Romanian Ministry of Research, Innovation and Digitization, CNCS -UEFISCDI, project number PN-IV-P1-PCE-2023-0264, within PNCDI IV. BH and RP acknowledge support from the CNRS International Research Network ECO-Math. We also thank our home institutions for hosting our reciprocal visits.

\end{document}